\documentclass[11pt,a4paper]{article}
\usepackage[DIV=11]{typearea}
\usepackage[english]{babel}
\usepackage[utf8x]{inputenc}
\usepackage{hyperref}
\usepackage{amsmath}
\usepackage{amssymb}
\usepackage{amsfonts}
\usepackage{amsthm}
\usepackage{amscd}
\usepackage{mathdots}
\usepackage{tikz}
\usepackage{pgfplots}
\usetikzlibrary{positioning}
\pgfplotsset{compat=1.8}
\usepackage{subfig}
\usepackage{fancyhdr}
\usepackage[noblocks]{authblk}
\usepackage{verbatim}
\makeatletter
\def\revddots{\mathinner{\mkern1mu\raise\p@
\vbox{\kern7\p@\hbox{.}}\mkern2mu
\raise4\p@\hbox{.}\mkern1mu\raise7\p@\hbox{.}\mkern1mu}}
\makeatother

\theoremstyle{plain}
\newtheorem{thm}{Theorem}[section]
\newtheorem{lem}[thm]{Lemma}
\newtheorem{prop}[thm]{Proposition}
\newtheorem{cor}[thm]{Corollary}

\theoremstyle{definition}
\newtheorem{defn}[thm]{Definition}
\newtheorem{conj}[thm]{Conjecture}

\theoremstyle{remark}

\newcommand{\R}{\mathbb{R}}

\newcommand{\C}{\mathbb{C}}

\newcommand{\Z}{\mathbb{Z}}

\newcommand{\Hom}{\mathrm{Hom}}

\newcommand{\SL}{\mathrm{SL}}
\newcommand{\GL}{\mathrm{GL}}
\newcommand{\Spin}{\mathrm{Spin}}

\newcommand{\Sp}{\mathrm{Sp}}
\newcommand{\SO}{\mathrm{SO}}
\newcommand{\Irr}{\mathrm{Irr}}
\newcommand{\zrep}[2]{\zeta(#1,#2)}

\DeclareMathOperator{\Jord}{Jord}

\newcommand\blfootnote[1]{%
  \begingroup
  \renewcommand\thefootnote{}\footnote{#1}%
  \addtocounter{footnote}{-1}%
  \endgroup
}
\title{On the Muić conjecture--the irreducibility of the big theta lift} 
\author{M. Hanzer}
\affil{Department of Mathematics, Faculty of Science, University of Zagreb}
\date{}

\begin{document}

\maketitle

\begin{abstract}
Let $F$ be a non-archimedean local field of characteristic zero. We study theta correspondence for (complex) representations of symplectic--even orthogonal dual reductive pairs over $F;$ more specifically, the big theta lifts. We prove that, starting from a discrete series representation $\pi$ of a symplectic (even orthogonal group) over $F,$ its big theta lift $\Theta (\pi)$ (as a representation  of an even orthogonal (symplectic) group) if non-zero, is an irreducible representation, thus proving a conjecture of Muić.  Building upon this result, we completely describe the situations in which the theta lifts of  tempered representations are irreducible and when they are  not.\blfootnote{\hspace{-1.8em}MSC2010: Primary 22E50, Secondary 11F27 \newline Keywords: theta correspondence; big theta lift; discrete series representation, tempered representation \newline This work was supported in part by the Croatian Science Foundation under the project number HRZZ-IP-2022-4615}
\end{abstract}
\section{Introduction}

In this paper, we study  some questions related to the local theta correspondence for dual pairs of type I. Let us briefly recall the basic setting before explaining our results.

Let $F$ be a nonarchimedean local field of of characteristic $0.$ We fix $\epsilon = \pm 1$. Let $W_n$ be an $-\epsilon$-hermitian space of dimension $n$ over $F$; similarly, let $V_m$ be a $\epsilon$-hermitian space of dimension $m$ over $F$.  To the pair of spaces $(W_n, V_m)$ we attach the corresponding isometry groups $G(W_n)$, $H(V_m).$
The space $W_n \otimes V_m$ has a natural symplectic structure. Fixing some additional data , we obtain a splitting $G(W_n) \times H(V_m) \to \text{Mp}(W_n \otimes V_m)$, where $\text{Mp}$ denotes the metaplectic cover of the symplectic group. By means of this splitting  we obtain a Weil representation of $G(W_n) \times H(V_m)$, which we denote $\omega_{m,n}$.

For any irreducible admissible representation $\pi$ of $G(W_n)$ we may now look at the maximal $\pi$-isotypic quotient of $\omega_{m,n}$. This quotient is isomorphic to $\pi \otimes \Theta(\pi)$ for a certain representation $\Theta(\pi)$ of $H(V_m)$. We call $\Theta(\pi)$ the big (or the full) theta lift of $\pi$ to $V_m$. This representation, when non-zero, has a unique irreducible quotient, denoted by $\theta(\pi)$---the small theta lift of $\pi$. This basic fact, called the Howe duality conjecture, was first formulated by Howe \cite{Howe_theta_series}, proven by Waldspurger \cite{Waldspurger_howe_duality} (for odd residue characteristic) and by Gan and Takeda \cite{Gan_Takeda_proof_of_Howe} in general. The Howe duality establishes a map $\pi \mapsto \theta(\pi)$, called the (local) theta correspondence.
The local theta correspondence has a large significance  in the local as well as in the global theory (where it agrees with local), where it presents (one of the few) direct ways of constructing automorphic forms.

\bigskip

Nowadays the local theta correspondence (i.e.~the determination of the small theta lift) is known explicitly, i.e.~if $\pi$ is given through its Langlands parameter, we know when exactly its small theta lift $\theta(\pi)$ is non-zero and we know the Langlands parameter (i.e~the Langlands quotient form of $\theta(\pi)$) cf.~\cite{AG_tempered}, \cite{BH_theta}, but there are number of questions still un-answered e.g.~in what extent the theta correspondence preserves the unitarity (cf.~\cite{Li_unitarity_stable},\cite{Li_low_rank}) or the Arthur parameterizations (cf.~\cite{Adams_A-conjecture}) etc. It turns out that these two questions are related in a very substantial way (\cite{Atobe_Minguez_Unitarity}).

\bigskip

One of the questions which is of a great importance in global applications requires from us to step back, i.e.~to try to derive the information about the whole big theta lift, i.e.~$\Theta(\pi).$  The question whether the big theta lift is irreducible (i.e.~equal to the small theta lift $\theta(\pi)$) is very important as it shows up  not only in the global theta correspondence, but also, more recently, in the context of the relative Langlands program (\cite{Sake_Venk_periods},\cite{Ben_Zvi_Sak_Venk_rel_Langlands}), where certain branching problems  (\cite{GPP_I}) turn out to be dual to the theta correspondence (as explained in the 10th and 12th section of \cite{Gan_notes_IHES}).

As an old (and fundamental) result in this direction, one can recall a well-known fact, proved by Kudla (\cite{Kudla_local_theta})-if $\pi$ is a supercuspidal representation, $\Theta(\pi)$ is irreducible, when non-zero.

\smallskip
 The initial point of our paper is a result of Muić obtained in \cite{Muic_Israel}, cf.~Theorem 6.2 there,  where he proves that in all the instances where the small theta lift $\theta(\pi)$  of  a discrete series $\pi$ is (non zero and) tempered, the big theta lift $\Theta(\pi)$ is irreducible.
We can make this explicit (on both going-down and going up tower-as treated in \cite{AG_tempered}). Muić conjectured in \cite{Muic_conjecture} that this facts holds for the discrete series even more generally, as we recall below.

\smallskip

In the archimedean situation, we have a result by Loke-Ma (\cite{Loke_Ma}):
when in stable range and $\pi$ unitary, $\Theta(\pi)$ is irreducible.
Rather tangential to the classical considerations we employ here, there are a number of instances of the different sorts of the exceptional theta correspondences (for groups over nonarchimedean fields)  where it turns out that the full lifts of certain classes of representations are also irreducible (cf.~\cite{Gan_Savin_except}); moreover in \cite{Hanzer_Savin} it is proved that the lifts of the unitarizable representations of $\SL(2,F)$ to $\SL(6,F),\Spin(12,F), E_7$ (as the dual pairs in  $E_6,E_7,E_8$ respectively) are irreducible.
These correspondences can also be viewed under umbrella of the relative Langlands correspondence (cf.~the second section of \cite{Mao_Zhang_Wan}).

After the submission of this paper, we've learned about Chen and Zou's paper in which they extend Loke-Ma result to the non-archimedean case (and a bit more).
There is a certain overlap with this paper (in the stable range), but not too much: they study a  kind of "generic situation" which completely avoids treating the case  which  most of this paper is devoted to.

\bigskip

In this  paper we prove a conjecture of Muić stated in \cite{Muic_conjecture}.
\begin{conj} Let $F$ be a non-archimedean local field of characteristic zero and $\pi$ an irreducible discrete series representation of $G(W_n).$ Then, its  big  theta lift $\Theta(\pi)$ to $H(V_m)$ is irreducible, if non-zero.
\end{conj}

\bigskip

The proof of the above conjecture is inductive over the rank of $G(W_n)$ and relays, essentially, on a very simple study of the derivatives of both the initial representation and its theta lift.
This simple idea covers most of the cases, but we're then left with the several cases for which we must use subtler arguments.

These arguments can be better understood if $\theta(\pi)$ is perceived as a representation of Arthur type (cf.~\cite{Arthur_endoscopic}) through the Adams conjecture (\cite{Adams_A-conjecture}, \cite{Moeglin_A_conjecture}) and then we use  some facts about the derivatives of the representations of Arthur type (cf.~\cite{Xu_M_parameterization}). This is the reason we stated our results only for even orthogonal/symplectic dual pair-we avoided the metaplectic/odd orthogonal dual reductive pair because, for the metaplectic groups, the results about  Arthur parameterization are still not complete (as we understand).

 We believe, though, that  our result is also true for that reductive dual pair-but we would have to adapt our arguments to completely avoid the use of Arthur parameters (which, we believe,  can be done, but it would make the proofs even more technical). 
\smallskip

Although Muić obtained some very important results about (i) the irreducibility of the lift when $n\ge m$ (ii) the  structure of the possible non-tempered subquotients in $\Theta(\pi)$ when $n<m$ (Theorem 4.1 of \cite{Muic_Israel}),
he was missing an essential part of (our) proof--namely, the proof that $\Theta(\pi)$ does not contain a square-integrable subquotients when $n<m.$

In the last part of the paper, building upon the results and techniques for the discrete series case, we examine the big theta lifts of an irreducible tempered representation, and determine precisely when they are irreducible and when they are not.

\bigskip

{\textbf{Acknowledgment}}\\ 

We want to thank Shaun Stevens and Justin Trias for a wonderful organization of  LMS Summer School on the theta correspondence, University of East Anglia, Norwich (UK) in July 2025. This paper growth out of the meetings at this summer school and realizing that an (relatively) older conjecture of Muić  fits nicely in the more complete picture of theta correspondence which we now  have, on  the one hand, and, on the other, a new point of view about it is given through connections with the determination of the unitary dual of classical groups  (\cite{Atobe_Minguez_Unitarity}) and the relative Langlands program. We thank Dipendra Prasad for drawing our attention to the paper \cite{Chen_Zou_small_big}.

\smallskip

We also want to thank Goran Muić and Marko Tadić for introducing us to the field and their support through the years.

\section{Preliminaries}
We collect several results we need for the proof.
\smallskip

We keep the notation from Introduction and use the notation from \cite{AG_tempered}: thus, $F$ is a non-archimedean local field of characteristic zero, $G(W_n) \times H(V_m)$ is a dual reductive pair consisting of a symplectic/even orthogonal group (or vice versa; we treat the two groups symmetrically). Let $\epsilon\in\{1,-1\}$ be such that $W_n$ is $-\epsilon$--hermitian space and $V_m$ $\epsilon$--hermitian space. We fix once and for all the splitting data  $G(W_n) \times H(V_m) \to \text{Mp}(W_n \otimes V_m)$
(this includes fixing a non-trivial additive character $\psi$ which determines the Weil representation of $\text{Mp}(W_n \otimes V_m)$), so that we have the Weil representation $\omega_{m,n}$ of $G(W_n) \times H(V_m)$ fixed (cf.~2.8 in \cite{AG_tempered}). Assume for a moment that $\epsilon=1,$ i.e.~$W_n$ is a symplectic space. When we fix a quadratic character $\chi_V$ (related to the discriminant of a quadratic space), there are two quadratic towers with that discriminant which are distinguished by, roughly, their Hasse-invariant, so we can denote these two orthogonal towers
by $\mathcal{V}^{+}_{\chi_V}$ and $\mathcal{V}^{-}_{\chi_V}$ (cf.~\cite{Kudla_notes} or \cite{AG_tempered}, the second section).  Thus, our $V_m$ belongs to one of those quadratic spaces. From now on, we use $G(W_n)$ and $G_n$ interchangeably; the same goes for $H(V_m)$ and $H_m.$ We denote by $\Irr(G_n)$ ($\Irr(H_m)$) the set of the equivalence classes of admissible, irreducible (complex) representations of $G_n$ ($H_m$, respectively).

\smallskip

There are some more notation/results we use:
\begin{enumerate}
\item For $\pi\in \Irr(G_n)$ and fixed $\chi_V$ and a tower, it is customary to examine the lift $\Theta(\pi)$ to different $H(V_m)'s,$ so we denote the big theta lift by $\Theta(\pi,H_m)$ to emphasize the group to which we lift. More often, when the tower to which we lift is fixed, we define $l=:n-m+\epsilon$ (recall that $m,n$ are the dimensions and not the Witt indexes) and denote
$\Theta(\pi,H_m)$ by $\Theta_l(\pi).$ So, the lifts to large groups $H_m$ correspond to negative $l$'s. We have an analogous notation for the small theta lift $\theta_{l}(\pi).$ 
\item For fixed $\chi_V,$ when we examine the lifts of a fixed $\pi$ to  $\mathcal{V}^{+}_{\chi_V}$ and $\mathcal{V}^{-}_{\chi_V}$,  $\pi$ appears earlier in theta correspondence with the groups in one tower than in the other; the tower in which $\pi$ appears earlier is called ``going-down'' tower and in which appears later is called
``going-up'' tower. If there is danger of confusion, we denote $\Theta_{l}(\pi)=\Theta_l^{down}(\pi)$ or $\Theta_l^{up}(\pi).$
\item  The corresponding indexes for the first appearances satisfy $l^{down}(\pi)+l^{up}(\pi)=-2$ (the \textit{conservation relation} in our context, cf.~the second section of \cite{AG_tempered}; recall that in our situation, $l$ defined above is an odd integer). This means that there is a possibility that $l^{down}(\pi)=l^{up}(\pi)=-1,$ so there is no really difference between  ``going-down'' and ``going-up'' tower for $\pi,$ i.e., the Langlands quotient expression for $\theta_{-l}(\pi)$ looks analogous in both towers (cf.~(2) and (4) of Theorem 4.3 of \cite{AG_tempered}). We sometimes use the notation $l(\pi)$ for $l^{down}(\pi);$ thus $l(\pi)\ge -1.$ 

\item We recall that when $\pi$ appears in theta correspondence with
a member of a fixed tower, then it appears with all the larger groups in that tower
(``the tower property''), and it always appears if we go far up in the tower (what is far enough?  e.g. when the Witt index of $V_m$ is greater or equal to the dimension of $W_n$--{\textit{the stable range}} situation).
\item If we fix $\pi \in \Irr(G_n),$ where $G_n$ is an even orthogonal group, then there is only one tower of symplectic groups, but the situation is fairly symmetric, because instead of lifting $\pi$ on two different towers, we simultaneously lift $\pi$ and $\pi\otimes \det;$ then, analogous results hold: the tower property, stable range,
the conservation relation with $l(\pi)+l(\pi\otimes det)=-2.$
 \item Recall (\cite{Arthur_endoscopic}) that a discrete series representation $\pi\in \Irr (G_n)$ can be described using its L-parameter $(\phi,\epsilon).$
 Here 
 \[\phi=\oplus_{i=1}^r\phi_r,\]
 is a direct sum of distinct irreducible orthogonal representations of the Weil-Deligne group $WD_F=W_F\times \SL(2,\C),$ so $\phi_i=\rho_i \otimes S_{a_i},$
 where $S_{a_i}$ is the unique algebraic representation of $ \SL(2,\C)$ of dimension $a_i,$ and $\rho_i$ can be identified, through LLC for $\GL(n,F)$, with an irreducible self-dual supercuspidal representation of $\GL(n_i,F)$ for some $n_i.$ The character $\epsilon$ is a character of the component group $A_{\phi}= \oplus_{i=1}^r\Z/2\Z a_i,$
 which is a $\Z/2\Z$--vector space with the canonical basis $\{a_i\}$ indexed by the summands $\phi_i$ of $\phi.$ To emphasize the correspondence, we write $(\phi,\epsilon)=(\phi_{\pi},\epsilon_{\pi});$ on the other hand, if $(\phi,\epsilon)$ are given (satisfying some extra conditions, cf.~8.~below), the corresponding discrete series $\pi$ is denoted $\pi=\pi(\phi,\epsilon).$ Actually, to each  representation $\pi\in \Irr(G_n)$ one can attach $(\phi_{\pi},\epsilon_{\pi});$ if $\pi$ is  non-tempered, $\rho_i$ are not necessarily unitarizable (this is LLC for classical groups as recalled in Appendix B of \cite{AG_tempered}). Sometimes, as $\epsilon_{\pi}$ attains values in the set $\{1,-1\}$,  we call it "a sign." \\ 

 The relation between the L-parameters of $\pi\in \Irr(G_n)$ (a discrete series representation) and $\theta_{l}(\pi)$ (not necessarily tempered representation) is given in the fifth section of \cite{AG_tempered}. We recall here that, if $H_m$ belongs to a tower $\mathcal{V}^{\pm}_{\chi_V}$ (the target), the only parts of the L-parameters of $\pi$ and $\theta_{-l}(\pi)$ which are mutually different, up to a twist by a character, are the ones which correspond to $\rho=\chi_V.$ The symplectic tower corresponds to the trivial character.
 \item If $G_n=\Sp(W_n),$ the representation $\phi$ must factor through $\SO(n+1,\C)$
 and if $G_n=O(V_n),$ the representation $\phi$ must factor through $O(n,\C)$; these are the duals groups $\widehat{G_n}$ in those instances.

 \item We define $\Phi(G_n)$ as the set of the representations $\phi$ above (parameterizing $\Irr (G_n)$), understood as the set of equivalence classes of isomorphisms to $\SO(n+1,\C)$ or  $O(n,\C)$.  The character $\epsilon$ has to satisfy certain conditions which we do not recall here, but are explained in Desideratum 3.1 of \cite{AG_tempered}.
 Let $\Phi_{temp}(G_n)$ denote a subset of {\textit{tempered}} parameters of $\Phi(G_n)$ which corresponds to the parameters for which each $\rho_i$ is unitarizable (supercuspidal) representation.
 Specifically, all the parameters $\phi_{\pi},$ for $\pi$ a
  discrete series representation, belong to  $\Phi_{temp}(G_n)$.

  \item Recall that, in our setting, A-parameter for $G_n$ is the $\widehat{G_n}$--conjugacy class of an admissible homomorphisms
 \[\psi:W_F\times \SL(2,\C)\times \SL(2,\C)\to \widehat{G_n}\]
 such that the image of $W_F$ is bounded. We can write
 \[\psi=\oplus_{\rho}\left(\oplus_{i\in I_{\rho}}\rho\otimes S_{a_i}\otimes S_{b_i}\right)\]
 where $\rho$ runs of the set of equivalence classes of irreducible unitarizable supercuspidal representations $\cup_{d\ge 1}Cusp_{unit}(\GL(d,F)).$
 All of the Arthur parameters we encounter in this paper will be of ``good parity'', meaning that each summand $\rho\otimes S_{a_i}\otimes S_{b_i}$ is of the same type as $\widehat{G_n};$ in our case, this means orthogonal.

\item The notion of the Jordan block of a discrete series representation was defined by M\oe glin, and was used to classify discrete series representation. We don't need the exact definition (one can check in \cite{MT}), but it turned out by the work of Arthur (\cite{Arthur_endoscopic}),  that it agrees with the L-parameter of $\pi$ (cf.~sections 12 and 13 of \cite{MT}). Thus, we can  define 
\[\Jord (\pi)=\{(\rho,a): \rho \otimes S_a\hookrightarrow \phi_{\pi}\}\]
so that $\rho$ is an irreducible  selfcontragredient cuspidal representation of $GL(m_{\rho},F)$--this defines $m_{\rho}$-- and $a$ is a positive integer such that $\rho \otimes S_a$ is of orthogonal type (for our dual reductive pairs).\\
For  such $\rho,$ let 
\[\Jord_{\rho}(\pi)=\{a: \rho \otimes S_a\hookrightarrow \phi_{\pi}\}.\]
We introduced the Jordan block  as above just as a matter of a notational convenience: namely to avoid repeating on numerous occasions that, say, summand $\rho\otimes S_a$ "appears as an irreducible summand in $\phi_{\pi}$", we just say $a\in \Jord_{\rho}(\pi).$
\item We denote the set of A-parameters  of $G_n$ by $\Psi(G_n),$ and the subset of the those of good parity by $\Psi_{gp}(G_n).$ By Arthur, to each $\psi$ one can attach a multiset (maybe empty) of $\Irr (G_n)$ (\cite{Arthur_endoscopic}) denoted $\Pi_{\psi}$--so called A-packet attached to $\psi.$ These representation are explicitly constructed by M\oe glin (\cite{Moeglin_construction}- she proved that they actually form a set; also cf.~\cite{Xu_M_parameterization}) and the parameterization is simplified by Atobe (\cite{Atobe_construction_A}).

\item To obtain a complete list of irreducible representations of $G(W_n)$, we use the Langlands classification. Let $\delta_i \in \Irr (GL(n_i,F)), i = 1,\dotsc, r$ be irreducible discrete series representations, and let $\tau$ be an irreducible tempered representation of $G(W_{n-2t})$, where $t=t_1+\dotsb+t_r$. Any representation of the form
\[
\nu^{s_r}\delta_r \times \dotsb \times \nu^{s_1}\delta_1 \rtimes \tau,
\]
where $s_r \geqslant \dotsb \geqslant s_1 > 0$ (and where $\nu$ denotes the character $\lvert\det\rvert$ of the corresponding general linear group) is called a standard representation (or a standard module)--here we use the Zelevinsky notation for the parabolic induction (cf.~the second section of \cite{AG_tempered}). It possesses a unique irreducible quotient, the Langlands quotient, denoted by $L(\nu^{s_r}\delta_r, \dotsc, \allowbreak \nu^{s_1}\delta_1; \allowbreak \tau)$. Conversely, every irreducible representation can be represented as the Langlands quotient of a unique standard representation.
\item Let $\rho \in \Irr (\GL(n,F))$ be a supercuspidal representation and $a,b\in \R$ such that $b-a\in\Z_{\ge 0}.$ Then $\{\rho\nu^{a},\rho\nu^{a+1},\ldots, \rho\nu^{b}\}$ is called a segment of supercuspidal representations. To it we attach a parabolically induced representation
\[\rho\nu^{b}\times \rho\nu^{b-1}\times \cdots \times \rho\nu^{a};\]
this representation has a unique irreducible subrepresentation denoted by $\delta[\rho\nu^{a},\rho\nu^{b}]$ and unique irreducible quotient which is denoted by $\zeta(\rho\nu^{a},\rho\nu^{b}).$ The representation $\delta[\rho\nu^{a},\rho\nu^{b}]$ is essentially square-integrable and each essentially square-integrable representation of a general linear group is obtained in this way.
\item To use Jacquet modules of representations which are parabolically induced, we use Tadić formula. To explain it, let us denote by $R_k(\pi)$ the normalized Jacquet module of a  smooth representation $\pi$ of $G_n$ with respect to the standard maximal parabolic subgroup of  $G_n$ with the Levi subgroup isomorphic to $\GL(k,F)\times G_{n-2k}.$ Then, we define (as a semisimplification)
\[\mu^*(\pi)=\sum_{k=0}^{n'}R_{P_k}(\pi),\]
where $n'$ is the split rank of $G_n.$

 The relevant formula is now
\begin{equation}
\label{eq_tadic_classical}
\mu^*(\delta \rtimes \pi) = M^*(\delta) \rtimes \mu^*(\pi).
\end{equation}
The definition of $M^*$ can be found in \cite[Theorem 5.4]{Tad_struc}, but we shall need it here only in the special case when $\delta = \delta[\nu^a\rho, \nu^b\rho]$ or $\zeta(\nu^a\rho, \nu^b\rho)$; in these cases, we have (\cite[\S 14]{tadic2012reducibility})
\begin{equation}
\label{eq_tadic_classical2}
\begin{aligned}
M^*(\delta[\rho\nu^a, \rho\nu^b]) &= \sum_{i=a-1}^b\sum_{j=i}^b \delta[{}^c\rho^\vee\nu^{-i},{}^c\rho^\vee\nu^{-a}] \times \delta[\rho\nu^{j+1},\rho\nu^b] \otimes \delta[\rho\nu^{i+1}, \rho\nu^j]\\
M^*(\zeta(\rho\nu^a, \rho\nu^b)) &=  \sum_{i=a-1}^b\sum_{j=i}^b \zrep{{}^c\rho^\vee\nu^{-b}}{{}^c\rho^\vee\nu^{-(j+1)}} \times \zrep{\rho\nu^{a}}{\rho\nu^{i}} \otimes \zrep{\rho\nu^{i+1}}{\rho\nu^j}.
\end{aligned}
\end{equation}

\item Let $\rho$ be a unitarizable supercuspidal representation of $\GL(k,F)$ and $\alpha \in \R.$ For an admissible representation $\pi$ of $G_n$ we denote (in the appropriate Grothendieck group)
\[R_{P_k}(\pi)=\rho\nu^{\alpha}\otimes \xi +\sum \tau\otimes \sigma,\]
where $\tau\neq \rho\nu^{\alpha}.$ Then, we denote $Jac_{\rho\nu^{\alpha}}(\pi)=\xi.$
This is {\textit{the first derivative of}} $\pi$ with respect to  $\rho\nu^{\alpha},$
for higher derivatives and relevant properties, one can check \cite{Atobe_Min_duality}.
\item  Let $\pi$ be an admissible representation of $G_n.$ Then, we denote the MVW-involution of $\pi$ by $\pi^{\eta}$ (cf.~\cite{MVW_theta}). 

\end{enumerate}

\smallskip
Several times we will need the notion of a lift in a more general context. Namely, if $M_k\cong \GL(k,F)\times G_{n_0}$ is a Levi subgroup attached to a
 maximal parabolic subgroup $P_k$ of $G_n$, then  the corresponding Jacquet module $R_{P_k}(\omega_{m,n})$ is a $\GL(k,F)\times G_{n_0}\times H_m$-module. Let  $\delta\otimes \sigma$ be an irreducible representation of $M_k.$ We can consider the maximal quotient of  $R_{P_k}(\omega_{m,n})$ on which $\GL(k,F)\times G_{n_0}$ acts as a multiple of $\delta\otimes \sigma.$ This is an $H_m$--module which we denote by
 \begin{equation}
 \label{theta_Jacquet}
 \Theta(\delta\otimes \sigma,R_{P_k}(\omega_{m,n})).
 \end{equation}

\smallskip

\begin{prop}
\label{Adams_for_square}
For each $l\ge 1,$ $\theta_{-l}^{down}(\pi)$ is an unitarizable representation of Arthur type. For each $l\ge l(\pi)+2,$ $\theta_{-l}^{up}(\pi)$ is  an unitarizable representation of Arthur type. Moreover, if  $\phi_{\pi}$ is the Langlands parameter of $\pi,$ then $\theta_{-l}(\pi)$ (on the going-down or going up tower depending on $l$ as just explained) belongs to the A-packet corresponding to the A-parameter given by
\[\psi_{\ldots}\chi_W\chi_V^{-1}\phi_{\pi}\oplus \chi_W\otimes S_1\otimes S_l.\]
Here, we understand each summand $\rho\otimes S_a$ in $\phi_{\pi}$ as of form
$\rho\otimes S_a\otimes S_1.$
\end{prop}

\begin{proof}This is an instance of Adams conjecture (cf.~\cite{Adams_A-conjecture}), and the claim of this Proposition was already known to M\oe glin-\cite{Moeglin_A_conjecture} and this generality is quite sufficient for our purpose. For a more general instance of Adams conjecture, one may check \cite{BH_Adams}
\end{proof}

 We note  Proposition \ref{Adams_for_square} since we use bellow Proposition 8.3 of \cite{Xu_M_parameterization} which is concerned with certain Jacquet modules of a representation in a given Arthur packet.

 \smallskip

 For our proof, we need the notion of an extended cuspidal support of a representation. We recall the notion of the extended cuspidal support as defined in \cite{Atobe_local_A_packets_fixed_rep}. 

\begin{defn}
Let $\pi \in Irr(G_n).$ Let $\rho_i \in Cusp(GL(d_i,F))\;1\le i\le r$ and $\sigma \in Irr(G_{n_0})$ such that
\[\pi\hookrightarrow \rho_1\times \rho_2\times \cdots \times \rho_k\rtimes \sigma.\] 
Let $\sigma=\pi(\phi_{\sigma},\epsilon_{\sigma})$  with
\[\phi_{\sigma}=\sum_{j=1}^t\rho_j'\otimes S_{a_j}\in \Phi_{temp}(G_{n_0})\]
We define an extended cuspidal support $ex.supp(\pi)$ as the multi-set
\[ex.supp(\pi) = \{\rho_1,\ldots,\rho_r,\widetilde{\rho_1},\ldots,\widetilde{\rho_r}\}\cup \bigsqcup_{i=1}^t\{\rho_j'\nu^{\frac{a_j-1}{2}},\rho_j'\nu^{\frac{a_j-3}{2}},\ldots, \rho_j'\nu^{-\frac{a_j-1}{2}}\}.\]
 \end{defn}

\smallskip
Let $\psi\in \Psi(G_n).$ Let $\Delta:W_F\times SL(2,\C)\to W_F\times SL(2,\C)\times SL(2,\C)$ be the diagonal embedding:
\[\Delta(w,x)=(w,x,x)\] and we define {\textit{the diagonal restriction of $\psi$}} in the notation $\psi_d,$ as $\psi_d:=\psi\circ \Delta.$

 The following Proposition is needed; this is Proposition 2.3. of \cite{Atobe_local_A_packets_fixed_rep}, and originally comes from Proposition 4.1 of 
 \cite{Moeglin_L_in_A}. 
 \begin{prop}
\label{extended_Adams}
 Let $\psi\in \Psi(G_n).$ Write $\psi_d=\sum_{i=1}^{r}\rho_i\otimes S_{a_i}.$ Then, for any $\pi\in \Pi_{\psi},$ the extended cuspidal support of $\pi$ is given by
 \[\bigsqcup_{i=1}^r\{\rho_i\nu^{\frac{a_i-1}{2}},\rho_i\nu^{\frac{a_i-3,}{2}},\ldots, \rho_i\nu^{-\frac{a_i-1}{2}}\}.\]
 \end{prop}
Thus, the representations in the same A-packet have the same extended cuspidal support, and, trivially, the representations with the same cuspidal support have the same extended cuspidal support.

The following Lemma is a direct consequence of Kudla's filtration (\cite{Kudla_local_theta}); also check the treatment in the fifth section of  \cite{AG_tempered}. We use it in the formulation of Corollary 4.6 of \cite{BH_theta}, also cf.~Corollary 4.7. there.
\begin{lem}
\label{AG-epimorphism}

Let $\pi\in \Irr(G_n), \pi_0 \in \Irr(G_{n-2k})$ and let $\delta$ be an irreducible essentially square integrable representation of $\GL_k(E)$. Assume that $\delta \ncong \textnormal{St}_k\nu^{\frac{l-k}{2}}$, where $l = n - m + \epsilon$. Then
\[
\chi_V\delta \rtimes \pi_0 \twoheadrightarrow \pi
\]
implies
\[
\chi_W\delta \rtimes \Theta_l(\pi_0) \twoheadrightarrow \Theta_l(\pi).
\]
If $\delta =\textnormal{St}_k\nu^{\frac{l-k}{2}}=\delta[\nu^{\frac{l+1}{2}-k},\nu^{\frac{l-1}{2}}]$ we have a bit weaker statement:  either
\[
\chi_W\delta \rtimes \Theta_l(\pi_0) \twoheadrightarrow \theta_l(\pi)
\]
or
\[
\chi_W\delta[\nu^{\frac{l+1}{2}-k},\nu^{\frac{l-3}{2}}] \rtimes \Theta_{l-2}(\pi_0) \twoheadrightarrow \theta_l(\pi)
\]
 holds.
 \end{lem}

\section{The main result}
Let $\pi \in \Irr(G_n)$ be a discrete series representation.
\begin{thm}
\label{going_down_tower}
The representation $\Theta_{-l}^{down}(\pi)$ is irreducible for $l\ge 3$ and 
$\Theta_{-l}^{up}(\pi)$ is irreducible for $l\ge l(\pi)+4.$

\end{thm}

\begin{proof} Recall that we know (\cite{Muic_Israel}) that  $\Theta_{-1}^{down}(\pi)$ and $\Theta_{-(l(\pi)+2)}^{up}(\pi)$ are irreducible.\\

\smallskip

 We argue by induction over $n$  and  assume that $\pi$ is not supercuspidal. We break the proof in three steps: first we prove that $\Theta_{-l}(\pi)$ as above does not have a discrete series subquotients, then that it does not have a tempered, non-square-integrable subquotient, and then that it does not have no non-tempered subquotients beside the small theta lift which appears in big theta lift with multiplicity one. From now on, we fix one tower to which we lift (i.e.~$\Theta_{-l}^{down}(\pi)$ or $\Theta_{-l}^{up}(\pi)$ ) and denote the lift by $\Theta_{-l}(\pi).$

A remark on the base case: at the base of the symplectic or the orthogonal towers lie the groups which are either finite or compact. Either way, their square-integrable representations are necessarily supercuspidal, so the main Theorem holds.

As mentioned above,  below we break the proof in three propositions (and some lemmas), each of them assuming the induction assumption over $n.$
\end{proof}

\begin{prop}
\label{square}
Under the induction assumption,  $\Theta_{-l}(\pi)$ does not have square-integrable subquotients.
\end{prop}

 \begin{proof}
 Assume the opposite; let $\xi_0$ be a square-integrable representation of $H_m$ such that $-l=n+\epsilon-m$ and $\xi_0\le \Theta_{-l}(\pi).$ It is known that $\xi_0$ and $\theta_{-l}(\pi)$ have the same cuspidal support (cf.~Lemma 4.2.~of \cite{Muic_Israel}). Recall that, if $l(\pi)>-1,$ we have  $\theta_{-l}^{down}(\pi)=L(\chi_W\nu^{\frac{l-1}{2}},\ldots,\chi_W\nu^{1};\theta_{-1}^{down}(\pi))$ and $\theta_{-l}^{up}(\pi)=L(\chi_W\nu^{\frac{l-1}{2}},\ldots,\chi_W\nu^{\frac{l(\pi)+3}{2}};\theta_{-l(\pi)-2}^{up}(\pi))$ (cf.~\cite{AG_tempered}, Section 4).
 If $l(\pi)=-1,$ the lifts on the both towers have  almost the same Langlands parameter (i.e.~Langlands quotient form), the only difference  being $\theta_{-1}(\pi)$ as lifts on the different towers.

This trivially implies that $\xi_0$ cannot be a cuspidal representation; more precisely, since $l\ge 3,$ it implies that $\xi_0$ cannot be, let's call it $\chi_W$--cuspidal,  namely, there exists some twist of $\chi_W,$ say $\chi_W\nu^{\frac{\beta-1}{2}}$ with $\beta \in \Z$ odd, such that $Jac_{\chi_W\nu^{\frac{\beta-1}{2}}}(\xi_0)\neq 0.$ Since $\xi_0$ is a discrete series representation, we necessarily have $\beta\ge 3.$ Let $\xi_0'$ be a tempered representation such that the following holds (this follows trivially from the parameter $\phi_{\xi_0}$)

\begin{equation}
\label{xi_o}
\xi_0\hookrightarrow \chi_W\nu^{\frac{\beta-1}{2}}\rtimes \xi_0'.
\end{equation}

\smallskip

As we mentioned above, since $\xi_0$ and $\theta_{-l}(\pi)$ have the same cuspidal support, they have the same extended cuspidal support. On the other hand, they are both representations of Arthur type (Proposition \ref{Adams_for_square}). This means that, if by $\psi$ we denote the A-parameter attached to $\theta_{-l}(\pi)$ by Proposition \ref{Adams_for_square}, we have $\phi_{\xi_0}=\psi_d$ by Proposition \ref{extended_Adams}; specifically, $\theta_{-l}(\pi)$ is a  representation with the discrete diagonal restriction (cf.~\cite{Xu_combinatorial}), meaning that $l\notin \Jord_{\chi_V}(\pi).$

\smallskip

Since $\pi$ is not cuspidal, there exists $\rho,$ an irreducible  selfcontragredient cuspidal representation of $GL(m_{\rho},F)$ and $\alpha \in\Z_{\ge 1}$ such that
\begin{equation}
\label{pi_embedd}
\pi \hookrightarrow \chi_V\nu^{\frac{\alpha-1}{2}}\rho\rtimes \pi',
\end{equation}
where $\pi'$ is a tempered representation whose L-parameter we easily read off from the L-parameter of $\pi.$ Note that $(\chi_V\rho, \alpha)\in \Jord(\pi)$ and  $l(\pi')\ge l(\pi).$ 
We need the following easy Lemma which we recall in future on numerous occasions:
\begin{lem}
\label{lift_non_tempered}
We retain the assumptions above. Recall that $l\notin \Jord_{\chi_V}(\pi).$
Then, $\theta_{-l}(\pi')$ is not tempered.
\end{lem}

\begin{proof} Assume that  $\theta_{-l}(\pi')$ is tempered. Since $l\ge 3,$ this happens only when $\theta_{-l}(\pi')$ is the first lift on the going-up tower for $\pi',$ i.e.~$l=l(\pi')+2\ge l(\pi)+2.$ 
If $l(\pi)\ge 1$ and if we look at $\theta_{-l}(\pi)=\theta_{-l}^{up}(\pi),$ (with $l>l(\pi)+2$), in order to $\theta_{-l}(\pi')$ be tempered, we need to have $l(\pi')>l(\pi).$ This happens only if $\pi'$ is square-integrable and $l(\pi)=\alpha-4$ so that $l(\pi')=\alpha-2$ (with correct values of  $\epsilon_{\pi}$). This means $l=\alpha,$ (with, of course, $\rho=1_{GL_1}$) but we saw that $\theta_{-l}(\pi)$ has to be a DDR-representation so  this is impossible.

\smallskip

 If $l(\pi)\ge 1$ and if we look at $\theta_{-l}(\pi)=\theta_{-l}^{down}(\pi),$ 
 then the ``going down'' towers for $\pi$ and $\pi'$ are the same, so that $\theta_{-l}(\pi')$ appearing in \eqref{first_epi} is $\theta_{-l}^{down}(\pi'),$ so $\theta_{-l}^{down}(\pi')$ is not tempered.

\smallskip

If $l(\pi)=-1,$ we can look at both of its lifts as the lifts on the ``going-down'' tower. Then,  it may happen that $\theta_{-l}(\pi')=\theta_{-(l(\pi')+2)}^{up}(\pi')$ is tempered, and we necessarily have $l(\pi')\ge 1.$ Thus $\rho=1_{GL_1},$
the minimal element of $\Jord_{\chi_V}(\pi)$ equals 3, with $\alpha=3,$ so that $l(\pi')=1$ and $\alpha=l,$ again, we have a contradiction. We conclude that, since $\alpha\neq l,$ $\theta_{-l}(\pi')$ is never tempered.
\end{proof}
1. Assume that $\rho\neq 1_{GL_1}$ or $\rho=1_{GL_1}$ and $\alpha\neq l+2.$
Then, we can use Lemma \ref{AG-epimorphism} to get that
\begin{equation}
\label{first_epi}
\chi_W\rho\nu^{-\frac{\alpha-1}{2}}\rtimes \Theta_{-l}(\pi')\twoheadrightarrow \Theta_{-l}(\pi).
\end{equation}
If $\pi'$ is square-integrable, we can argue by the inductive assumption that $\Theta_{-l}(\pi')=\theta_{-l}(\pi').$ If $\pi'$ is tempered, but not square-integrable, the same thing holds; indeed, then 
\[\pi'\hookrightarrow \chi_V\delta[\nu^{-\frac{\alpha-3}{2}}\rho, \nu^{\frac{\alpha-3}{2}}\rho]\rtimes \pi'',\]
where $\pi''$ is a discrete series representation. We can apply $\Theta_{-l}$ with Lemma \ref{AG-epimorphism}, together with the inductive assumption to get
\[\chi_W\delta[\nu^{-\frac{\alpha-3}{2}}\rho, \nu^{\frac{\alpha-3}{2}}\rho]\rtimes \theta_{-l}(\pi'')\twoheadrightarrow \Theta_{-l}(\pi').\]
Since $\theta_{-l}(\pi'')$ is unitarizable, the left hand side is a semi-simple representation, so is $ \Theta_{-l}(\pi'),$ necessarily making it irreducible by the Howe duality conjecture (cf.~\cite{Gan_Takeda_Howe_duality}).  

\smallskip

From \eqref{xi_o} and \eqref{first_epi}, we get that
\[\chi_W\nu^{\frac{\beta-1}{2}}\otimes \xi_0'\le \mu^{*}(\chi_W\nu^{\frac{\alpha-1}{2}}\rho\rtimes \theta_{-l}(\pi')).\]

Assume first that $\rho=1_{GL_1}$ and $\alpha=\beta.$ Since $\pi$ is square--integrable, $Jac_{\chi_V\nu^{\frac{\alpha-1}{2}}}(\pi')=0,$ so   $Jac_{\chi_W\nu^{\frac{\alpha-1}{2}}}(\theta_{-l}(\pi'))=0.$ Tadić formula \eqref{eq_tadic_classical} then gives 
$\xi_0'=\theta_{-l}(\pi'),$ but that is impossible since $\xi_0'$ is tempered, and $\theta_{-l}(\pi')$ is not, by Lemma \ref{lift_non_tempered}.

\smallskip

If $\rho\neq 1_{GL_1}$ or $\rho=1_{GL_1},$ but $\alpha\neq \beta,$ by \eqref{eq_tadic_classical} we get that $Jac_{\chi_W\nu^{\frac{\beta-1}{2}}}(\theta_{-l}(\pi'))\neq 0.$ From Lemma \ref{xi_o}, we get that $Jac_{\chi_V\nu^{\frac{\beta-1}{2}}}(\pi')\neq 0.$ Thus,
\[\pi\hookrightarrow \chi_V\nu^{\frac{\alpha-1}{2}}\rho\times \chi_V\nu^{\frac{\beta-1}{2}}\rtimes \pi'',\]
where $\pi''$ is a tempered representation. If $\rho\neq 1_{GL_1},$ or $\rho= 1_{GL_1},$ but $\frac{\alpha-1}{2}\neq \frac{\beta-1}{2}\pm 1$, we can interchange two factors, so that again $Jac_{\chi_V\nu^{\frac{\beta-1}{2}}}(\pi)\neq 0$ and we can write $\pi\hookrightarrow \chi_V\nu^{\frac{\beta-1}{2}}\rtimes \pi_0'$ for some other tempered $\pi'_0$ and repeat the argument from the first part of the proof to get $\xi_0'=\theta_{-l}(\pi_0'),$
which again can not hold.

\smallskip

The case that remained is when $\rho=1_{GL_1}$ and $\beta=\alpha \pm 2.$

\smallskip

a) Let  $\rho=1_{GL_1}$ and $\beta=\alpha + 2.$
The problematic case is when we $Jac_{\chi_V\frac{\beta-1}{2}}(\pi)=Jac_{\chi_V\frac{\alpha+1}{2}}(\pi)=0$ (but $Jac_{\chi_V\frac{\alpha+1}{2}}(\pi')\neq 0$). Recall that $\pi'=Jac_{\chi_V\frac{\alpha-1}{2}}(\pi).$
This appears precisely when we, in $\phi_{\pi},$ have summands
\[\chi_V\otimes S_{\alpha}\oplus \chi_V\otimes S_{\alpha+2}\]
with $\epsilon_{\pi}(\chi_V\otimes S_{\alpha})=-\epsilon_{\pi}(\chi_V\otimes S_{\alpha+2})$ (and if $\chi_V\otimes S_{\alpha-2}$ is in  $\phi_{\pi},$ it has the same sign as  $\chi_V\otimes S_{\alpha}$). Recall that we examine
\[\chi_W\nu^{\frac{\alpha+1}{2}}\otimes \xi_0'\le \mu^*(\chi_W\nu^{\frac{\alpha-1}{2}}\rtimes \theta_{-l}(\pi')).\]
Let $\pi''=Jac_{\chi_V\nu^{\frac{\alpha+1}{2}}}(\pi');$ this is an irreducible representation by Proposition 8.3 of \cite{Xu_M_parameterization} combined with Section 3.2 of \cite{Atobe_Min_duality}. Then, by the induction assumption, since we assume $\alpha\neq l$ (because $(\chi_V,l)\notin \Jord(\pi)$) $\theta_{-l}(\pi'')=Jac_{\chi_W\nu^{\frac{\alpha+1}{2}}}(\theta_{-l}(\pi')).$ By \eqref{eq_tadic_classical}
\[\xi_0'\le \chi_W\nu^{\frac{\alpha-1}{2}}\rtimes \theta_{-l}(\pi'').\]

We claim that the representation $\chi_W\nu^{\frac{\alpha-1}{2}}\rtimes \theta_{-l}(\pi'')$  does not have tempered subquotients.

\bigskip

 The argument can be slightly simplified when $\pi''$ is a discrete series, since then $\chi_V\nu^{\frac{\alpha-1}{2}}\rtimes \pi''$ is irreducible, as follows  from  the discussion after Lemma 5.3 in \cite{MT}, but we here give a unified argument when $\pi''$ is tempered (discrete series or not). Note that since we have here that $\alpha,\alpha+2 \in \Jord_{\chi_V}(\pi)$ then, necessarily, $l\neq \alpha,\alpha+2$ (by our assumption that $\Theta_{-l}(\pi)$ has a discrete series subquotient). By examining the extended cuspidal support of $\pi'',$ we see that the only discrete series/tempered subquotient
of $\chi_W\nu^{\frac{\alpha-1}{2}}\rtimes \theta_{-l}(\pi'')$ is necessarily tempered, and a subrepresentation of a representation 
\[\delta[\chi_W\nu^{-\frac{\alpha-1}{2}},\chi_W\nu^{\frac{\alpha-1}{2}}]\rtimes \lambda\]
for some tempered representation $\lambda.$ This means that
\begin{align*}
&\delta[\chi_W\nu^{-\frac{\alpha-1}{2}},\chi_W\nu^{\frac{\alpha-1}{2}}]\otimes \lambda \le \\ 
&\mu^{*}(\chi_W\nu^{\frac{\alpha-1}{2}}\rtimes \theta_{-l}(\pi''))= (\chi_W\nu^{\frac{\alpha-1}{2}}\otimes 1+\chi_W\nu^{-\frac{\alpha-1}{2}}+1\otimes\chi_W\nu^{\frac{\alpha-1}{2}})\rtimes \mu^*(\theta_{-l}(\pi'')).
\end{align*}
There exists an irreducible $\xi_1\otimes \xi_2\le \mu^*(\theta_{-l}(\pi''))$
such that 
\[\delta[\chi_W\nu^{-\frac{\alpha-1}{2}},\chi_W\nu^{\frac{\alpha-1}{2}}]\le A\times \xi_1,\]

where $A\in \{\chi_W\nu^{\frac{\alpha-1}{2}},\chi_W\nu^{-\frac{\alpha-1}{2}},1\}.$
Assume first that $A=\chi_W\nu^{\frac{\alpha-1}{2}}.$ 
Then, it follows that
\[\delta[\chi_W\nu^{-\frac{\alpha-1}{2}},\chi_W\nu^{\frac{\alpha-3}{2}}]\otimes \xi_2\le \mu^*(\theta_{-l}(\pi''))\]
and
\[\theta_{-l}(\pi'')\hookrightarrow \chi_W\nu^{\frac{\alpha-3}{2}}\times  \chi_W\nu^{\frac{\alpha-5}{2}}\times \cdots\times\chi_W\nu^{-\frac{\alpha-1}{2}}\rtimes \xi_2', \]
for some irreducible $\xi_2'.$

Now we employ the procedure used by Muić (\cite{Muic_Israel}, Theorem 4.1).
We ``cut'' the segment $[-\frac{\alpha-1}{2}, \frac{\alpha-3}{2}]$ into $k$ subsegments--
$[-\frac{\alpha-1}{2}, \frac{\alpha-3}{2}]=[-\frac{\alpha-1}{2},\gamma_1]\cup [\gamma_1+1,\gamma_2]\cup\cdots \cup  [\gamma_{k-1}+1,\frac{\alpha-3}{2}] $ such that 
\[\theta_{-l}(\pi'')\hookrightarrow \chi_W\delta[\nu^{\gamma_{k-1}+1},\nu^{\frac{\alpha-3}{2}}]\times \cdots \times \chi_W\delta[\nu^{-\frac{\alpha-1}{2}},\nu^{\gamma_1}].\] holds.
Note that such segments surely exist; our initial embedding corresponds to the case where all the segments are singletons. We chose the smallest possible $k$.
Known facts about reducibility (in the general linear group) of two essentially square integrable representations attached to segments insure that the minimality of $k$ enables  that, in that case, $\theta_{-l}(\pi'')$ can be embedded in any   induced representation obtained  by any permutation of those  representations (essentially discrete series). Thus, there exists an irreducible representation $\lambda_1$ such that
\[\theta_{-l}(\pi'')\hookrightarrow \chi_W\delta[\nu^{-\frac{\alpha-1}{2}},\nu^{\gamma_1}]\rtimes \lambda_1.\]
Since $l\neq \alpha,$ using Lemma \ref{AG-epimorphism}, we get
\[\pi''\hookrightarrow \chi_V\delta[\nu^{-\frac{\alpha-1}{2}},\nu^{\gamma_1}]\rtimes \widetilde{\lambda_1}^{\eta}.\]
Since $\gamma_1\le \frac{\alpha-3}{2},$ this contradicts the temperedness of $\pi''.$
\smallskip
Assume now that $A=\chi_W\nu^{-\frac{\alpha-1}{2}}$ or $A=1.$ Then, it follows that $Jac_{\chi_W\nu^{\frac{\alpha-1}{2}}}(\theta_{-l}(\pi''))\neq 0,$ so $Jac_{\chi_V\nu^{\frac{\alpha-1}{2}}}(\pi'')\neq 0,$ a contradiction with the values of $\epsilon_{\pi''}.$ In this way we have proved that $\chi_W\nu^{\frac{\alpha-1}{2}}\rtimes \theta_{-l}(\pi'')$  does not have tempered subquotients.

\bigskip

b) Let  $\rho=1_{GL_1}$ and $\beta=\alpha - 2.$ The problematic cases are the following
\begin{itemize}
\item[(i)]	 $\Jord_{\chi_V}(\pi)$ contains $\alpha,$ but not $\alpha-2$ nor $\alpha-4$
\item[(ii)] $\Jord_{\chi_V}(\pi)$  contains $\alpha,$ does not contain $\alpha-2$ but contains  $\alpha-4;\;\epsilon_{\pi}(\chi_V\otimes S_{\alpha})=\epsilon_{\pi}(\chi_V\otimes S_{\alpha-4})$ 
\item[(iii)]  $\Jord_{\chi_V}(\pi)$  contains $\alpha,\alpha-2, \alpha-4;\,\epsilon_{\pi}(\chi_V\otimes S_{\alpha})=\epsilon_{\pi}(\chi_V\otimes S_{\alpha-2})=-\epsilon_{\pi}(\chi_V\otimes S_{\alpha-4}).$ 
\end{itemize}
(i) Since we assumed $l\neq \alpha-2,$ when we examine the extended cuspidal support of $\chi_W\nu^{\frac{\alpha-1}{2}}\rtimes \theta_{-l}(\pi''),$ we see that there cannot be any tempered subquotients in this induced representation (since the summand $\chi_W\otimes S_{\alpha-2}$ is missing).

(ii) Again, as in the previous case, since $l\neq {\alpha-2},$ the extended cuspidal support of $\chi_W\nu^{\frac{\alpha-1}{2}}\rtimes \theta_{-l}(\pi'')$ cannot carry a tempered representation.

(iii)
Let $\lambda_0$ be u unique square-integrable representation such that
\[\pi\hookrightarrow \chi_V\delta[\nu^{-\frac{\alpha-3}{2}},\nu^{\frac{\alpha-1}{2}}]\rtimes \lambda_0.\] Recall that, in our situation, we have
 \[\xi_0\hookrightarrow\chi_W\nu^{\frac{\alpha-3}{2}}\rtimes \xi_0',\]
 with $\xi_0'$ tempered, so that, with the induction assumption $\Theta_{-l}(\lambda_0)=\theta_{-l}(\lambda_0)$ we get
 \[\chi_W\nu^{\frac{\alpha-3}{2}}\otimes \xi_0'\le \mu^{*}(\chi_W\delta[\nu^{-\frac{\alpha-3}{2}},\nu^{\frac{\alpha-1}{2}}]\rtimes \theta_{-l}(\lambda_0)).\]
 
 By \eqref{eq_tadic_classical} and \eqref{eq_tadic_classical2}, we have
 \[\chi_W\nu^{\frac{\alpha-3}{2}}\le\chi_W\delta[\nu^{-i},\nu^{\frac{\alpha-3}{2}}]\times \chi_W\delta[\nu^{j+1},\nu^{\frac{\alpha-1}{2}}]\times \xi_1 \]
 and
 \[\xi_0'\le \chi_W\delta[\nu^{i+1},\nu^{j}]\rtimes \xi_2,\]
 where $-\frac{\alpha-1}{2}\le i\le j\le \frac{\alpha-1}{2}$ and $\xi_1\otimes \xi_2$ is an irreducible subquotient of $\mu^{*}(\theta_{-l}(\lambda_0)).$
 Since $Jac_{\chi_V\nu^{\frac{\alpha-3}{2}}}(\lambda_0)=Jac_{\chi_W\nu^{\frac{\alpha-3}{2}}}(\theta_{-l}(\lambda_0))=0,$ we see that $j=\frac{\alpha-1}{2},\; \xi_1=1,\;i=-\frac{\alpha-3}{2}.$ This means

\[\xi_0'\le \chi_W\delta[\nu^{-\frac{\alpha-5}{2}},\nu^{\frac{\alpha-1}{2}}]\rtimes \theta_{-l}(\lambda_0).\]
Recall that $l\neq \alpha-4, \alpha-2, \alpha.$ From the extended cuspidal support of $\lambda_0$  we can conclude that there exists a discrete series representation $\xi_0''$ such that
\[\xi_0'\hookrightarrow \chi_W\delta[\nu^{-\frac{\alpha-5}{2}},\nu^{\frac{\alpha-5}{2}}]\rtimes \xi_0'' \] so that
\[\chi_W\delta[\nu^{-\frac{\alpha-5}{2}},\nu^{\frac{\alpha-5}{2}}]\otimes \xi_0''\le \mu^*(\chi_W\delta[\nu^{-\frac{\alpha-5}{2}},\nu^{\frac{\alpha-1}{2}}]\rtimes \theta_{-l}(\lambda_0)).\]
Using \eqref{eq_tadic_classical} again, we get that 
\[\chi_W\delta[\nu^{-\frac{\alpha-5}{2}},\nu^{\frac{\alpha-5}{2}}]\le \chi_W\delta[\nu^{-i_1},\nu^{\frac{\alpha-5}{2}}]\times \xi_1',\]
and
\[\xi_0''\le \chi_W\delta[\nu^{i_1+1},\nu^{\frac{\alpha-1}{2}}]\rtimes \xi_2', \]
where $\xi_1'\otimes \xi_2'\le \mu^*(\theta_{-l}(\lambda_0))$ and $-\frac{\alpha-3}{2}\le i_1\le \frac{\alpha-1}{2}.$ Assume firstly that $\xi_1'\neq 1;$ we now prove that this is impossible. Indeed, in that case $\xi_1'=\chi_W\delta[\nu^{-\frac{\alpha-5}{2}},\nu^{-i_1-1}]$ with $-i_1-1\le \frac{\alpha-5}{2};$ we prove that this is impossible. Indeed, we then have
\[\theta_{-l}(\lambda_0)\hookrightarrow \chi_W\nu^{-i_1-1}\times \chi_W\nu^{-i_1}\times \cdots \times \nu^{-\frac{\alpha-5}{2}}\rtimes \xi_2'',\]
for some irreducible $\xi_2''.$
Now we again employ the procedure used by Muić (\cite{Muic_Israel}, Theorem 4.1), in the same way we have done it in a previous case. We obtain that  there exists an irreducible representation $\lambda_1$ such that
\[\theta_{-l}(\lambda_0)\hookrightarrow \chi_W\delta[\nu^{-\frac{\alpha-5}{2}},\nu^{\gamma_1}]\rtimes \lambda_1.\]
Since we have that $l\neq \alpha-4,$ Lemma \ref{AG-epimorphism} gives
\[\lambda_0\hookrightarrow \chi_W\delta[\nu^{-\frac{\alpha-5}{2}},\nu^{\gamma_1}]\rtimes \widetilde{\Theta_{l}(\lambda_1)}^{\eta}.\]
Since $\gamma_1\le -i_1-1\le \frac{\alpha-5}{2},$ this contradicts the square-integrability of $\lambda_0.$ 

We conclude that $\xi_1'=1.$ This means that
\[\xi_0''\le \chi_W\delta[\nu^{\frac{\alpha-3}{2}},\nu^{\frac{\alpha-1}{2}}]\rtimes \theta_{-l}(\lambda_0). \]
On the other hand, we know that $ \alpha\in \Jord_{\chi_W}(\xi_0'')$ but $\alpha-2,\alpha-4\notin \Jord_{\chi_W}(\xi_0'').$  This means that 
\[\xi_0''\hookrightarrow \chi_W\nu^{\frac{\alpha-1}{2}}\rtimes \xi_0''',\]
where $\xi_0'''$ is square-integrable. This means
\[\chi_W\nu^{\frac{\alpha-1}{2}}\otimes \xi_0'''\le \mu^*(\chi_W\delta[\nu^{\frac{\alpha-3}{2}},\nu^{\frac{\alpha-1}{2}}]\rtimes \theta_{-l}(\lambda_0)).\]
Again, we easily get that $\xi_0'''\le \chi_W\nu^{\frac{\alpha-3}{2}}\rtimes \theta_{-l}(\lambda_0),$ since $Jac_{\chi_W\nu^{\frac{\alpha-1}{2}}}(\theta_{-l}(\lambda_0))=0.$  Analogously, $\xi_0'''\hookrightarrow \chi_W\nu^{\frac{\alpha-3}{2}}\rtimes \xi_0^{iv},$
where $\xi_0^{iv}$ is square-integrable. This means $\xi_0^{iv}=\theta_{-l}(\lambda_0)$ a contradiction, since $\theta_{-l}(\lambda_0)$ is not square-integrable. Indeed, $l(\lambda_0)\le l(\pi),$ and $\lambda$ and $\pi$ share the same ``going-down'' tower. So, if $\theta_{-l}(\lambda_0)=\theta_{-l}^{down}(\lambda_0)$ this is obvious, and if $\theta_{-l}(\lambda_0)=\theta_{-l}^{up}(\lambda_0)$ then note that $l>l(\pi)+2\ge l(\lambda_0)+2$ and the claim follows.

2. Assume now that the only possibility for \eqref{pi_embedd} is with  $\rho=1_{GL_1}$ and $\alpha=l+2$:
 \begin{equation}
 \label{bad_case}
 \pi\hookrightarrow \chi_V\nu^{\frac{l+1}{2}}\rtimes \pi',
 \end{equation}
 where $\pi'$ is square-integrable (since $l\notin \Jord_{\chi_V}(\pi)$); here $l(\pi)\le l(\pi')$.
 Further, it means that $\{1,3,\ldots,r\}\in \Jord_{\chi_V}(\pi)$
   with $r\le l-2,$ and $\epsilon_{\pi}$ alternates on this set (without gaps); the value of $\epsilon_{\pi}$ on the part of $\phi_{\pi}$ with $\chi_V\otimes S_{l+2}\oplus \cdots$ also alternates without gaps.

   Note that we can't apply Lemma \ref{AG-epimorphism} with $\Theta_{-l}$ on the equation \eqref{bad_case}, but we still have some information on $\Theta_{-l}(\pi).$  Indeed, since we, by \eqref{bad_case}, have 
   \[R_{P_1}(\omega_{n,n+\epsilon+l})\twoheadrightarrow \chi_V\nu^{\frac{l+1}{2}}\otimes \pi'\otimes \Theta_{-l}(\pi)\]
   it follows that $\Theta(\chi_V\nu^{\frac{l+1}{2}}\otimes \pi', R_{P_1}(\omega_{n,m}))\twoheadrightarrow \Theta_{-l}(\pi)$ (cf.~\eqref{theta_Jacquet}). 

   \smallskip
   On the other hand, we have information on $\Theta(\chi_V\nu^{\frac{l+1}{2}}\otimes \pi', R_{P_1}(\omega_{n,m}))$ by \cite{Muic_Israel}, Corollary 3.2.
   It has a filtration of the form
   \[0\subset \Theta_0\subset \Theta(\chi_V\nu^{\frac{l+1}{2}}\otimes \pi', R_{P_1}(\omega_{n,m})),\]
   where 
   \begin{align*}
   &\chi_W\nu^{-\frac{l+1}{2}}\rtimes \Theta_{-l}(\pi')\twoheadrightarrow \Theta_0\\ 
   &\Theta_{-(l+2)}(\pi')\twoheadrightarrow\Theta(\chi_V\nu^{\frac{l+1}{2}}\otimes \pi', R_{P_1}(\omega_{n,m}))/\Theta_0.
   \end{align*}

  By the induction hypothesis $\Theta_{-l}(\pi')=\theta_{-l}(\pi')$ and $\Theta_{-(l+2)}(\pi')=\theta_{-(l+2)}(\pi').$ Since $\theta_{-(l+2)}(\pi')\hookrightarrow \chi_W\nu^{-\frac{l+1}{2}}\rtimes \theta_{-l}(\pi'),$ we get that all the irreducible subquotients appearing in $\Theta(\chi_V\nu^{\frac{l+1}{2}}\otimes \pi', R_{P_1}(\omega_{n,m}))$ are the subquotients of $\chi_W\nu^{-\frac{l+1}{2}}\rtimes \theta_{-l}(\pi'),$ maybe with some multiplicities. This means that
  \begin{equation}
  \label{eq:0}
  \xi_0\le \chi_W\nu^{\frac{l+1}{2}}\rtimes \theta_{-l}(\pi');
  \end{equation}
  an analogous relation  to the one  we have exploited before.

  Let $\beta$ be, as before, such that $Jac_{\chi_W\nu^{\frac{\beta-1}{2}}}(\xi_0)\neq 0$ and $\xi_0\hookrightarrow \chi_W\nu^{\frac{\beta-1}{2}}\rtimes \xi_0',$
  for tempered $ \xi_0'.$  This gives
\[\chi_W\nu^{\frac{\beta-1}{2}}\otimes \xi_0'\le \mu^*(\chi_W\nu^{\frac{l+1}{2}}\rtimes \theta_{-l}(\pi')).\]
  Now, if $\beta=l+2,$ we get, as before, that $\xi_0'=\theta_{-l}(\pi'),$ a contradiction with Lemma \ref{lift_non_tempered}. If $\beta\neq l+2$ we have, by \eqref{eq_tadic_classical}, $Jac_{\chi_W\nu^{\frac{\beta-1}{2}}}(\theta_{-l}(\pi'))\neq 0$ and  by Lemma \ref{AG-epimorphism}, we get $Jac_{\chi_V\nu^{\frac{\beta-1}{2}}}(\pi')\neq 0.$ So, the only possibility is $\beta=l+4$ if $\chi_V\otimes S_{l+4}$ is a summand on $\phi_{\pi},$ or $\beta=l$ if $r<l-2$ or $r=l-2$ and $\epsilon_{\pi}(\chi_V\otimes S_{l-2})=\epsilon_{\pi}(\chi_V\otimes S_{l+2}).$

  \smallskip

   Assume that $\beta=l$  and $r=l-2$ with the signs on $\chi_V\otimes S_{l-2}$ and on $\chi_V\otimes S_{l+2}$ in $\phi_{\pi}$ agreeing. Let $\pi'':=Jac_{\chi_V\nu^{\frac{l-1}{2}}}(\pi')\neq 0$ be an irreducible tempered representation (cf.~\cite{Atobe_Min_duality}).  We get
  \begin{equation}
\label{eq1}
  \xi_0'\le \chi_W\nu^{\frac{l+1}{2}}\rtimes \theta_{-l}(\pi'').
  \end{equation}
  By the extended cuspidal support, we get that
  \begin{equation}
  \label{eq2}
  \xi_0'\hookrightarrow \chi_W\delta[\nu^{-\frac{l-3}{2}},\nu^{\frac{l-3}{2}}]\rtimes \lambda_1,
  \end{equation}
  for some square integrable $\lambda_1.$ This means that, for some irreducible $\xi_2,$ 
  \begin{equation}
\label{eq3}
  \chi_W\delta[\nu^{-\frac{l-3}{2}},\nu^{\frac{l-3}{2}}]\otimes \xi_2\le \mu^*( \theta_{-l}(\pi''))
  \end{equation}
  and
  \begin{equation}
  \label{eq4}
  \lambda_1\le \chi_W\nu^{\frac{l+1}{2}}\rtimes \xi_2.
  \end{equation}
  On the other hand, we have $\pi''\hookrightarrow \chi_V\delta[\nu^{-\frac{l-3}{2}},\nu^{\frac{l-3}{2}}]\rtimes \pi''',$ for a square-integrable $\pi'''.$
  Now, \eqref{eq_tadic_classical}, together with the fact that $Jac_{\chi_W\nu^{\gamma}}(\theta_{-l}(\pi'''))=0,$ for all $|\gamma|\le \frac{l-3}{2}$ (this follows from our assumptions on $\pi$-the alternating signs), gives that $\xi_2=\theta_{-l}(\pi''').$ Now, just by keeping track of the L-parameter of $\lambda_1,$ it turns out that there is a square-integrable $\lambda_2$ such that
  \[\lambda_1\hookrightarrow \chi_W\nu^{\frac{l+1}{2}}\rtimes \lambda_2.\]
  The last two displays then give $\theta_{-l}(\pi''')=\lambda_2,$ a contradiction with the square-integrability of $\lambda_2.$ Indeed, we know exactly (the relevant part) of L-parameter of $\pi''';$ we have $l(\pi''')=l-4$
  so $\theta_{-l}(\pi''')$ can't be square-integrable (as a lift on either of the towers, even in a somewhat degenerate case of $l=3$).

\smallskip

Assume now that  $\beta=l$ and $r<l-2.$ We still have \eqref{eq1}, but now both $\xi_0'$ and $\pi''$ are square-integrable representations, moreover $\xi_0'\hookrightarrow\chi_W\nu^{\frac{l+1}{2}}\rtimes \xi_0'',$ for a square-integrable $\xi_0'',$ so that 
\[\chi_W\nu^{\frac{l+1}{2}}\otimes \xi_0''\le \mu^*(\chi_W\nu^{\frac{l+1}{2}}\rtimes \theta_{-l}(\pi'')),\]
which gives $\xi_0''=\theta_{-l}(\pi'').$ Now, if $\theta_{-l}(\pi'')=\theta_{-l}^{up}(\pi''),$ it can happen that it is a quadratic representation. Now assume that it is the case; this means that $l=l(\pi'')+2,$ and this forces $r=l-4,$ and $\epsilon_{\pi''}(\chi_V\otimes S_{l-4})=-\epsilon_{\pi''}(\chi_V\otimes S_{l-2})=-\epsilon_{\pi}(\chi_V\otimes S_{l+2}).$ This determines $\epsilon_{\xi_0''}$ (by \cite{AG_tempered}), and, consequently,  $\epsilon_{\xi_0}.$
We now prove that this is impossible. Note that
\[ \epsilon_{\xi_0''}(\chi_W\otimes S_{l-2})=-\epsilon_{\xi_0''}(\chi_W\otimes S_{l}),\]
so that
\[\epsilon_{\xi_0}(\chi_W\otimes S_{l})=-\epsilon_{\xi_0}(\chi_W\otimes S_{l+2}),\]
since we may assume that $Jac_{\chi_W\nu^{\frac{l+1}{2}}}(\xi_0)=0$ (as $\beta=l+2$ is covered in the previous case).
On the other hand,
\[\theta_{-l}(\pi')=\chi_W\nu^{\frac{l-1}{2}}\rtimes \xi_0''=\chi_W\nu^{\frac{l-1}{2}}\rtimes\theta_{-l}^{up}(\pi'').\]
Indeed, the representation on the right hand side is irreducible by Section 5 of \cite{MT} and we know, from Lemma \ref{AG-epimorphism}, that the left-hand side is a subquotient of the right-hand side. We are interested in the composition series of a representation $\chi_W\nu^{\frac{l+1}{2}}\rtimes \theta_{-l}(\pi'),$ because of \eqref{eq:0}.

\smallskip

We claim that the representation (the equality is in the appropriate Grothendieck group)
\[\chi_W\nu^{\frac{l+1}{2}}\rtimes \theta_{-l}(\pi')=\chi_W\nu^{\frac{l+1}{2}}\times \chi_W\nu^{\frac{l-1}{2}} \rtimes \theta_{-l}^{up}(\pi'')\]
is of length 4, and multiplicity free.

\smallskip
Note that the representation
\[\chi_W\nu^{\frac{l+1}{2}}\rtimes \theta_{-l}^{up}(\pi'')\]
is of length 2. Indeed, if there are no elements larger then $l-2$ in $\Jord_{\chi_V}(\pi''),$ then the representation $\theta_{-l}^{up}(\pi'')$ is cuspidal and the claim is well known-the subquotients are the  the quadratic representation and its Aubert dual, say, $\lambda$ and $\widehat{\lambda}$, and we know that the representation $\chi_W\nu^{\frac{l-1}{2}}\rtimes \lambda$ is of length two.
All in all, we get that the length of 
\[\chi_W\nu^{\frac{l+1}{2}}\times \chi_W\nu^{\frac{l-1}{2}} \rtimes \theta_{-l}^{up}(\pi'')\]
equals four, and all the subquotients are representations of Arthur type with a parameter easily constructed from the L-parameter $\phi_{ \theta_{-l}^{up}(\pi'')}.$ On the other hand, when $\theta_{-l}^{up}(\pi'')$ is not cuspidal, i.e.
there is part $\chi_V\otimes S_{l+4}\oplus\cdots $ (no gaps, alternating) we can, essentially, reason analogously.
The representation 
\[\chi_W\nu^{\frac{l+1}{2}}\rtimes \theta_{-l}^{up}(\pi'')\]
is again of length two. Now we have to elaborate a bit more: using the reasoning similar to the proof of Lemma 2.2 of \cite{Muic_positive} and the fact that $\theta_{-l}^{up}(\pi'')$ is strongly positive we conclude that this representation does not have any non-tempered subquotients beside the Langlands quotient. The only other possible subquotients are the square-integrable subquotients. We know that it has an obvious subrepresentation, say, $\lambda$ such that
$\epsilon_{\lambda}=\epsilon_{\theta_{-l}^{up}(\pi'')}$ except 
$\epsilon_{\lambda}(\chi_W\otimes S_{l+2})=\epsilon_{\theta_{-l}^{up}(\pi'')}(\chi_W\otimes S_{l}).$ Let $\lambda_1$ be any other square-integrable subquotient. Note that $\phi_{\lambda}=\phi_{\lambda_1},$ the only difference is in the epsilon function. However, note that $Jac_{\chi_W\nu^{\frac{l+1}{2}}}(\lambda_1)\neq 0.$ We calculate the multiplicity of all the factors of the form
$\chi_W\nu^{\frac{l+1}{2}}\otimes \xi_2$ in $\mu^{*}(\chi_W\nu^{\frac{l+1}{2}}\rtimes \theta_{-l}^{up}(\pi'')).$ We get that necessarily $\xi_2=\theta_{-l}^{up}(\pi'')$ and the multiplicity is one. Thus, it is of length two. Note that both subquotients ($\lambda$ and $L(\chi_W\nu^{\frac{l+1}{2}};\theta_{-l}^{up}(\pi''))$) are representations of Arthur type with easy describable A-parameters-e.g. an A-parameter of $L(\chi_W\nu^{\frac{l+1}{2}};\theta_{-l}^{up}(\pi''))$ has the same summands as the parameter $\phi_{\lambda},$ except that it has $\chi_W\otimes S_1\otimes S_{l+2}$ instead of $\chi_W\otimes S_{l+2}\otimes S_1.$

\smallskip

We can prove that $\chi_W\nu^{\frac{l-1}{2}}\rtimes \lambda$ is of length two using totally analogous reasoning as above. A little bit of more effort leads to the proof that $\chi_W\nu^{\frac{l-1}{2}}\rtimes L(\chi_W\nu^{\frac{l+1}{2}};\theta_{-l}^{up}(\pi''))$ is of length two. Again, all four representations are of Arthur type and mutually different.

\smallskip
On the other hand, in the appropriate Grothendieck group, we have
\begin{align*}
&\chi_W\nu^{\frac{l+1}{2}}\times \chi_W\nu^{\frac{l-1}{2}} \rtimes \theta_{-l}(\pi'')=\\ 
&\chi_W\delta[\nu^{\frac{l-1}{2}},\nu^{\frac{l+1}{2}}]\rtimes \theta_{-l}(\pi'')+
\chi_W\zeta(\nu^{\frac{l-1}{2}},\nu^{\frac{l+1}{2}})\rtimes \theta_{-l}(\pi'').
\end{align*}
By our construction, we have that
\[\pi\hookrightarrow \chi_V\delta[\nu^{\frac{l-1}{2}},\nu^{\frac{l+1}{2}}]\rtimes \pi'',\]
so that Lemma \ref{AG-epimorphism} then gives
\[\xi_0\le \chi_W\delta[\nu^{\frac{l-1}{2}},\nu^{\frac{l+1}{2}}]\rtimes \theta_{-l}^{up}(\pi'').\]
On the other hand, by our construction (alternate signs!), we know that
\[\xi_0\hookrightarrow \chi_W\zeta(\nu^{\frac{l-1}{2}},\nu^{\frac{l+1}{2}})\rtimes \xi_0''=\chi_W\zeta(\nu^{\frac{l-1}{2}},\nu^{\frac{l+1}{2}})\rtimes\theta_{-l}^{up}(\pi'').\]
Since, as we have seen above, the representations $\chi_W\delta[\nu^{\frac{l-1}{2}},\nu^{\frac{l+1}{2}}]\rtimes \theta_{-l}(\pi'')$ and $\chi_W\zeta(\nu^{\frac{l-1}{2}},\nu^{\frac{l+1}{2}})\rtimes \theta_{-l}(\pi'')$ do not have common subquotients, we get a contradiction.

\smallskip

Assume now that $\beta=l+4.$ We can repeat our argument of a previous  case 1.a)
 with $\beta=\alpha+2.$
\end{proof}

\begin{prop}
Under the induction assumption,  $\Theta_{-l}(\pi)$ does not have tempered, non-square-integrable subquotients.
\end{prop}

\begin{proof}
We tweak and expand arguments used in Proposition \ref{square}.
Assume that $\xi_0\le \Theta_{-l}(\pi)$ is tempered, non-square-integrable representation. This forces $l\in \Jord_{\chi_V}(\pi)$ and
\[\xi_0\hookrightarrow \chi_W\delta[\nu^{-\frac{l-1}{2}},\nu^{\frac{l-1}{2}}]\rtimes T_0,\]
where $T_0$ is a square-integrable representation whose L-parameter (but not the $\epsilon$-character) is easily related to that of $\pi.$ Let $\xi_0'$ be an irreducible representation such that
\[\xi_0\hookrightarrow \chi_W\nu^{\frac{l-1}{2}}\rtimes \xi_0';\]
this is what we use for \eqref{xi_o} so that $\beta=l.$  Note that $\xi_0'$
\[\xi_0'\le \chi_W\delta[\nu^{-\frac{l-1}{2}},\nu^{\frac{l-3}{2}}]\rtimes T_0\]
and that  $\xi_0' $ is either tempered (if $l-2\in \Jord_{\chi_V}(\pi)$), square-integrable (if $l-2\notin \Jord_{\chi_V}(\pi)$) or $\xi_0'=L(\chi_W\delta[\nu^{-\frac{l-3}{2}},\nu^{\frac{l-1}{2}}];T_0).$ Indeed, we have the following Lemma. 
\begin{lem}
\label{structure}
Let $\lambda$ be a square-integrable representation of $G_n$ and such that $l\notin \Jord_{\rho}(\lambda),$ where $\rho$ is is a selfdual, irreducible supercuspidal representation of $GL_k(F)$ and $l$ of good parity (with respect to $\rho$ and $G_n$). Then, a standard representation
\[\delta[\rho\nu^{-\frac{l-3}{2}},\rho\nu^{\frac{l-1}{2}}]\rtimes \lambda\]
is
\begin{itemize}
\item[(i)] is of length three, if $l-2\notin \Jord_{\rho}(\lambda).$
There are two square-integrable subquotients and the Langlands quotient.
\item[(ii)] is of length at most two, if $l-2 \in \Jord_{\rho}(\lambda).$
Besides the Langlands quotient, there is at most one tempered subquotient.
(We prove below that this representation is, actually,  of length two).
\end{itemize}	
\end{lem}

\begin{proof}The first claim follows (and is a part of) M\oe glin-Tadić classification of the discrete series (\cite{MT}). For the second claim we use Lemma 2.2 of \cite{Muic_positive} to see that there are no other non-tempered subquotients other then the Langlands quotient. From the extended cuspidal support we see that any tempered subquotient $T$ embeds as 
\[T\hookrightarrow\delta[\rho\nu^{-\frac{l-3}{2}},\rho\nu^{\frac{l-3}{2}}]\rtimes \lambda',\]
for a square-integrable $\lambda'.$ The relation \eqref{eq_tadic_classical} then gives that $\lambda'\le \rho\nu^{\frac{l-1}{2}}\rtimes \lambda,$ and this representation has  a unique discrete series subquotient. All in all, this gives that there is (at most) one tempered subquotient in $\delta[\rho\nu^{-\frac{l-3}{2}},\rho\nu^{\frac{l-1}{2}}]\rtimes \lambda.$
\end{proof}

Later on in the proof we shall need the following result.

\begin{lem}
\label{alternation}
Let $T$ be a tempered representation such that $\phi_T$ is multiplicity free, except that $\rho\otimes S_l$ appears with multiplicity 2. Assume that $\rho\otimes S_{l-2}$ also appears in $\phi_T$ and $\epsilon_{T}(\rho\otimes S_l)=-\epsilon_{T}(\rho\otimes S_{l-2}).$  Then the highest derivative of $T$ with respect to $\rho\nu^{\frac{l-1}{2}}$ is of order one and $D^{(1)}_{\rho\nu^{\frac{l-1}{2}}}(T)=L(\delta[\rho\nu^{-\frac{l-3}{2}},\rho\nu^{\frac{l-1}{2}}];\sigma),$
where $\sigma$ is such that $T\hookrightarrow \delta[\rho\nu^{-\frac{l-1}{2}},\rho\nu^{\frac{l-1}{2}}]\rtimes \sigma$ holds.
\end{lem}

\begin{proof}
Let $T'$ be a tempered representation such that 
\[T\oplus T'=\delta[\rho\nu^{-\frac{l-1}{2}},\rho\nu^{\frac{l-1}{2}}]\rtimes \sigma.\] Thus, $\epsilon_{T'}(\rho\otimes S_l)=\epsilon_{T'}(\rho\otimes S_{l-2}).$  We calculate when the irreducible representation $\rho\nu^{\frac{l-1}{2}}\times \rho\nu^{\frac{l-1}{2}}\otimes \xi$ (for some $\xi$) appears in $\mu^*(\delta[\rho\nu^{-\frac{l-1}{2}},\rho\nu^{\frac{l-1}{2}}]\rtimes \sigma).$
We get $\xi=\delta[\rho\nu^{-\frac{l-3}{2}},\rho\nu^{\frac{l-3}{2}}]\rtimes \sigma$ and, moreover $D^{(2)}_{\rho\nu^{\frac{l-1}{2}}}(\delta[\rho\nu^{-\frac{l-1}{2}},\rho\nu^{\frac{l-1}{2}}]\rtimes \sigma)=\delta[\rho\nu^{-\frac{l-3}{2}},\rho\nu^{\frac{l-3}{2}}]\rtimes \sigma$ (which is an irreducible representation). This means that $D^{(2)}_{\rho\nu^{\frac{l-1}{2}}}(T)=0,$ and $D^{(1)}_{\rho\nu^{\frac{l-1}{2}}}(T)$ is an irreducible representation (\cite{Atobe_Min_duality}, Section 3.2). We have
\begin{align*}
&T\hookrightarrow \delta[\rho\nu^{-\frac{l-1}{2}},\rho\nu^{\frac{l-1}{2}}]\rtimes \sigma\\ 
&\hookrightarrow \rho\nu^{\frac{l-1}{2}}\times \delta[\rho\nu^{-\frac{l-1}{2}},\rho\nu^{\frac{l-3}{2}}]\rtimes \sigma.
\end{align*}
Let $\xi\le \delta[\rho\nu^{-\frac{l-1}{2}},\rho\nu^{\frac{l-3}{2}}]\rtimes \sigma$ be an irreducible subquotient contributing in the embedding of $T.$ If $\delta[\rho\nu^{-\frac{l-1}{2}},\rho\nu^{\frac{l-3}{2}}]\rtimes \sigma$ were irreducible, or reducible, but the tempered subquotient of it would be $\xi,$ we would have
\[T\hookrightarrow \rho\nu^{\frac{l-1}{2}}\times \delta[\rho\nu^{-\frac{l-3}{2}},\rho\nu^{\frac{l-1}{2}}]\rtimes \sigma,\]
so that
\[T\hookrightarrow \rho\nu^{\frac{l-1}{2}}\times \rho\nu^{\frac{l-1}{2}}\rtimes \cdots,\]
a contradiction. Additionally, this guarantees that in Lemma \ref{structure}, in the second case, the representation is of length two.

 \end{proof}
Now we follow the line of proof of Proposition \ref{square}; we use the same notation.

\noindent 1. Assume that $\rho\neq 1_{GL_1}$ or $\rho=1_{GL_1}$ and $\alpha\neq l+2.$ \\ 

First again assume additionally that  $\rho=1_{GL_1}$ and $\alpha=\beta=l.$
We get that $\xi_0'=\theta_{-l}(\pi').$ Now we discuss why this is impossible.
If $l-2\in \Jord_{\chi_V}(\pi),$ then  $\epsilon_{\pi}(\chi_V\otimes S_{l-2})=\epsilon_{\pi}(\chi_V\otimes S_{l}),$ the representation  $\pi'$ is tempered and non-square integrable and 
$\pi$ and $\pi'$ share the same going-down tower since $l(\pi)=l(\pi').$ If $\theta_{-l}(\pi')=\theta_{-l}^{down}(\pi'),$ from the explicit form $\theta_{-l}^{down}(\pi')$ we see that it cannot be equal to $\xi_0'.$ On the other hand, when $\theta_{-l}(\pi')=\theta_{-l}^{up}(\pi'),$ we can have $\theta_{-l}^{up}(\pi')=L(\chi_W\delta[\nu^{-\frac{l-3}{2}},\nu^{\frac{l-1}{2}}];T_0)=\xi_0'$, but then  we must have $l=l(\pi')+2=l(\pi)+2.$ In that case,  easily follows that $\xi_0=\theta_{-l}^{up}(\pi)$ but this cannot happen by our assumption on $l$ in Theorem \ref{going_down_tower}.   On the other hand, if $l-2\notin \Jord_{\chi_V}(\pi),$ we have $l(\pi')\ge l(\pi).$ If $l(\pi)\ge 1,$ then (since $\pi$ and $\pi'$ have the same ``going down tower''), if $\theta_{-l}(\pi)=\theta_{-l}^{down}(\pi),$
then $\theta_{-l}(\pi')=\theta_{-l}^{down}(\pi')$ can't be tempered.
If, on the other hand,  if $\theta_{-l}(\pi)=\theta_{-l}^{up}(\pi),$ and we assumed that $l\ge l(\pi)+4,$ it can happen that $\theta_{-l}^{up}(\pi')$ is the first lift on the going-up tower for $\pi'$ if $l=l(\pi')+2=l(\pi)+4.$
Recall that although $Jac_{\chi_W\nu^{\frac{l-1}{2}}}(\xi_0)$ is not irreducible, we still have \eqref{xi_o}, where $\xi_0'$ is a square-integrable representation with the same signs on $\chi_W\otimes S_{l-2}$ and $\chi_W\otimes S_{l}.$
But then $\xi_0'$ can't be equal to $\theta_{-l}^{up}(\pi'),$ since this representation has the opposite signs on $\chi_W\otimes S_{l-2}$ and $\chi_W\otimes S_{l}.$
Now we discuss the case of  $l(\pi)=-1.$ If $l>3$ the discussion is the same as the case with lifting both $\pi$ and $\pi'$ on the going-down towers. If, on the other hand, have $l=3,$ then $l(\pi')=1.$
Then, the lifts $\theta_{-3}(\pi)$ can be treated as $\theta_{-3}^{down}(\pi)$  for both $\epsilon$--towers, but,  for a choice of $\epsilon$, lift    on that tower is the $\theta_{-3}^{up}(\pi')$--the first lift on the going up tower for $\pi';$ a square-integrable representation. But this square-integrable representation is not one of the square-integrable subquotients of $\chi_W\delta[\nu^{-\frac{l-3}{2}},\nu^{\frac{l-1}{2}}]\rtimes T_0=\chi_W\delta[\nu^{0},\nu^{1}]\rtimes T_0$ as the signs directly show.
\smallskip

 Now we consider the case $\rho\neq 1_{GL_1}$ or $\rho=1_{GL_1},$ but $\alpha\neq l.$ Again, by the same argument as in Proposition \ref{square}, we resolve the case $\rho\neq 1_{GL_1},$ or $\rho= 1_{GL_1},$ but $\frac{\alpha-1}{2}\neq \frac{\beta-1}{2}\pm 1.$ We are left with the following cases.

\smallskip
a) Let  $\rho=1_{GL_1}$ and $l=\alpha + 2;$ this corresponds to the case 1.a) of Proposition \ref{square}. We can follow the argument completely.

\smallskip

 This proves that the representation $\chi_W\nu^{\frac{l-3}{2}}\rtimes \theta_{-l}(\pi'')$ does not have tempered subquotients. In this Proposition, we still have to address the case $\xi_0'=L(\chi_W\delta[\nu^{-\frac{l-3}{2}},\nu^{\frac{l-1}{2}}];T_0).$ We can argue in a very similar way as in the tempered case; i.e.~we prove that we can't have
 \[\chi_W\delta[\nu^{-\frac{l-1}{2}},\nu^{\frac{l-3}{2}}]\otimes T_0\le \mu^{*}(\chi_W\nu^{\frac{l-3}{2}}\rtimes \theta_{-l}(\pi'')).\]
 Using the same notation as in the tempered case, we prove that the above cannot hold by using the temperedness of $\pi''$ in the case $A=\chi_W\nu^{\frac{l-3}{2}}.$ Here we can't use Lemma \ref{AG-epimorphism} directly, but we use the filtration of Corollary 3.2 of \cite{Muic_Israel} to the same effect. In cases $A=\chi_W\nu^{-\frac{l-3}{2}}$ and $A=1$ we use the fact that $Jac_{\chi_V\nu^{\frac{l-3}{2}}}(\pi'')=0$   to prove that the last display can't hold.

\noindent b) Let  $\rho=1_{GL_1}$ and $l=\alpha - 2.$ We are in the case 2. of Proposition \ref{square}, so we get that L-parameter of $\pi$ is as described in that case (no gaps, alternating signs), just $\chi_V\otimes S_l$ is now a summand in $\phi_{\pi};$ there are no gaps between the summands, and the only two consecutive summands with the same sign are $\chi_V\otimes S_l$ and  $\chi_V\otimes S_{l+2}.$ Now $\pi'=Jac_{\chi_V\nu^{\frac{l+1}{2}}}(\pi)$ is a tempered representation with $\epsilon_{\pi'}(\chi_V\otimes S_{l-2})=-\epsilon_{\pi'}(\chi_V\otimes S_{l}).$  If, in $\xi_0,$ we have the same signs on $2\chi_W\otimes S_l$ and $\chi_W\otimes S_{l+2}$ we would have $\xi_0'':=Jac_{\chi_W\nu^{\frac{l+1}{2}}}(\xi_0)\neq 0$ tempered, and, again, $\xi_0''=\theta_{-l}(\pi'),$ a contradiction. Indeed, here $l=l(\pi')=l(\pi)$ so necessarily $\theta_{-l}(\pi')=\theta_{-l}^{down}(\pi').$ So, we can assume that in  $\xi_0,$ we have  $\epsilon_{\xi_0}(\chi_W\otimes S_l)=-\epsilon_{\xi_0}(\chi_W\otimes S_{l+2}).$

\smallskip

 By Lemma \ref{alternation}, $Jac_{\chi_V\nu^{\frac{l-1}{2}}}(\pi')$ is irreducible and equal to $L(\chi_V\delta[\nu^{-\frac{l-3}{2}},\nu^{\frac{l-1}{2}}];\pi_0),$ where $\pi_0$ is a  square integrable representation such that
\[\pi'\hookrightarrow \chi_V\delta[\nu^{-\frac{l-1}{2}},\nu^{\frac{l-1}{2}}]\rtimes \pi_0.\] We denote $\pi''=L(\chi_V\delta[\nu^{-\frac{l-3}{2}},\nu^{\frac{l-1}{2}}];\pi_0).$

\smallskip
 
We claim that $Jac_{\chi_W\nu^{\frac{l-1}{2}}}(\theta_{-l}(\pi'))$ is irreducible and equal to $\theta_{-l}(\pi'').$ Indeed, assume that 
\[\chi_W\nu^{\frac{l-1}{2}}\times \chi_W\nu^{\frac{l-1}{2}}\otimes \lambda\le \mu^*(\theta_{-l}(\pi'))\]
for some irreducible $\lambda.$ Using Lemma \ref{AG-epimorphism}, we would get that $D^{(2)}_{\chi_V\nu^{\frac{l-1}{2}}}(\pi')\neq 0,$ a contradiction, according to Lemma \ref{alternation}. Thus, again, following \cite{Atobe_Min_duality}, we conclude that $Jac_{\chi_W\nu^{\frac{l-1}{2}}}(\theta_{-l}(\pi'))$ is irreducible. Using this fact together with
\[\theta_{-l}(\pi')\hookrightarrow \chi_W\nu^{\frac{l-1}{2}}\rtimes (\widetilde{\Theta_{-l}(\pi'')}^{\eta}\]
followed by Frobenius reciprocity, we conclude that
\[\theta_{-l}(\pi')\hookrightarrow \chi_W\nu^{\frac{l-1}{2}}\rtimes \theta_{-l}(\pi'').\]

\smallskip

Since $\xi_0$ is not a square-integrable representation any more, as in Proposition \ref{square}, the situation with its derivatives is a bit more complicated.\\ 
 Assume, firstly, that the summands $\chi_W\otimes S_{l-2}$ and $2\chi_W\otimes S_{l}$ are of the same sign (as the representation $T'$ in Lemma \ref{alternation}). We easily get that $D^{(1)}_{\chi_V\nu^{\frac{l-1}{2}}}(\xi_0)$ 
is not irreducible, but we can pick $\xi_0'$ to be the tempered representation with the summands $2\chi_W\otimes S_{l-2}\oplus \chi_W\otimes S_{l}$
(with the same sign as in $\xi_0$, the rest of the L-parameter is the same as in $\xi_0$) since
\[\xi_0\hookrightarrow \chi_W\nu^{\frac{l-1}{2}}\rtimes \xi_0'.\]
We, thus, have that \eqref{eq1}, \eqref{eq2}, \eqref{eq3} and \eqref{eq4} hold.
Now we modify the argument a little bit. Namely, by Lemma \ref{AG-epimorphism} and the definition of $\pi_0$ (recall that $\pi''$ is not tempered as in Proposition \ref{square}), we have
\[\theta_{-l}(\pi'')\hookrightarrow \chi_W\delta[\nu^{-\frac{l-1}{2}}, \nu^{\frac{l-3}{2}}]\rtimes \theta_{-l}(\pi_0),\]
so that, together with \eqref{eq3}, we get
\[ \chi_W\delta[\nu^{-\frac{l-3}{2}}, \nu^{\frac{l-3}{2}}]\otimes \xi_2\le \mu^*(\chi_W\delta[\nu^{-\frac{l-1}{2}}, \nu^{\frac{l-3}{2}}]\rtimes \theta_{-l}(\pi_0)).\]
Tadić formula \eqref{eq_tadic_classical}, together with the structure of $\pi_0$ (as in corresponding part of the proof of Proposition \ref{square}), gives
\[\lambda_1\le \chi_W\nu^{\frac{l-1}{2}}\times \chi_W\nu^{\frac{l+1}{2}}\rtimes \theta_{-l}(\pi_0).\]
By examining the extended cuspidal support of $\lambda_1,$ we conclude that there exists a square-integrable $\lambda_2$ such that 
\[\lambda_1\hookrightarrow \chi_W\nu^{\frac{l-1}{2}}\rtimes \lambda_2.\]
Since $l,l+2\notin \Jord a_{\chi_V}(\pi_0),$ we conclude 
\[\lambda_2\le \chi_W\nu^{\frac{l+1}{2}}\rtimes \theta_{-l}(\pi_0),\]
and since 
\[\lambda_2\hookrightarrow \chi_W\nu^{\frac{l+1}{2}}\rtimes \lambda_3\]
for a square-integrable $\lambda_3,$
we conclude that $\lambda_3=\theta_{-l}(\pi_0),$ a contradiction. Indeed, note that the conditions on $\phi_{\pi}$ and $\epsilon_{\pi}$ give that $l(\pi)=l,$
so that we necessarily look at $\theta_{-l}(\pi)=\theta_{-l}^{down}(\pi).$ Also, $l(\pi_0)=l-2\ge 1$ and $\pi$ and $\pi_0$ have the same going-down tower so that $\theta_{-l}(\pi_0)=\theta_{-l}^{down}(\pi_0),$ a non-tempered representation.

\smallskip

Secondly, we discuss the situation in which  the summands $\chi_W\otimes S_{l-2}$ and $2\chi_W\otimes S_{l}$ in $\xi_0$ are of the opposite signs, so that $\xi_0'=L(\chi_W\delta[\nu^{-\frac{l-3}{2}},\nu^{\frac{l-1}{2}}];T_0),$ where $T_0$ is a square-integrable representation defined in the beginning of the proof of this Proposition. To remind on the situation with the functions $\epsilon_{\xi_0}$ and $\epsilon_{\pi},$ recall that we now have:
\[\pi \ldots \chi_V\oplus S_1\oplus \cdots\overset{-\epsilon}{\chi_V\otimes S_{l-2}}\oplus \overset{\epsilon}{\chi_V\otimes S_{l}}\oplus \overset{\epsilon}{\chi_V\otimes S_{l+2}}\oplus \cdots\]
\[\xi_0\ldots \chi_W\oplus S_1\oplus \cdots\overset{-\epsilon_1}{\chi_W\otimes S_{l-2}}\oplus \overset{\epsilon_1}{2\chi_W\otimes S_{l}}\oplus \overset{-\epsilon_1}{\chi_W\otimes S_{l+2}}\oplus \cdots\]
for some $\epsilon,\epsilon_1\in\{1,-1\}.$
We have that \eqref{eq1} holds with this $\xi_0',$ but, instead of \eqref{eq2}, we have
\[\chi_W\delta[\nu^{-\frac{l-1}{2}},\nu^{\frac{l-3}{2}}]\otimes T_0\le \mu^*(\chi_W\nu^{\frac{l+1}{2}}\rtimes \theta_{-l}(\pi'')).\]
This gives that
\[T_0\le \chi_W\nu^{\frac{l+1}{2}}\rtimes \xi_2,\]
where $\chi_W\delta[\nu^{-\frac{l-1}{2}},\nu^{\frac{l-3}{2}}]\otimes \xi_2\le (\theta_{-l}(\pi'')).$
Tadić formula gives
\[\chi_W\delta[\nu^{-\frac{l-1}{2}},\nu^{\frac{l-3}{2}}]\le \chi_W\delta[\nu^{-i},\nu^{\frac{l-3}{2}}]\times \chi_W\delta[\nu^{j+1},\nu^{\frac{l-1}{2}}]\times \xi_1'\]
and
\[\xi_2\le \chi_W\delta[\nu^{i+1},\nu^{j}]\rtimes \xi_2', \]
where $\xi_1'\otimes \xi_2'\le \mu^{*}(\theta_{-l}(\pi_0)).$
This means that $\xi_1'=\chi_W\delta[\nu^{-\frac{l-1}{2}},\nu^{-i-1}],$
for  $-i-1\le \frac{l-3}{2}.$ From the Adams-Arthur parameter of $\theta_{-l}(\pi_0)$ (Proposition \ref{Adams_for_square} and \cite{Xu_M_parameterization}) it follows that either  $\xi_1'=1$ or $\xi_1'=\chi_W\nu^{-\frac{l-1}{2}}.$ 

\smallskip

Assume firstly that $\xi_1'=1.$ Then, $\xi_2'=\theta_{-l}(\pi_0)$ with $i=j=\frac{l-1}{2}.$
This means $\xi_2=\theta_{-l}(\pi_0),$ so that we have
\[T_0\le \chi_W\nu^{\frac{l+1}{2}}\rtimes \theta_{-l}^{down}(\pi_0).\]
Again, as in several instances above, we have $T_0\hookrightarrow \chi_W\nu^{\frac{l+1}{2}}\rtimes T_1,$ with $T_1$ square integrable. From this follows that $T_1=\theta_{-l}(\pi_0),$ a contradiction with the fact that $\theta_{-l}(\pi_0)=\theta_{-l}^{down}(\pi_0),$ as just above.

\smallskip
Assume now that $\xi_1'=\chi_W\nu^{-\frac{l-1}{2}};$ then $\xi_2'=\theta_{-(l-2)}(\pi_0).$ Analogously as above, we now get
\[T_0\le \chi_W\nu^{\frac{l+1}{2}}\times \chi_W\nu^{\frac{l-1}{2}}\rtimes \theta_{-(l-2)}^{down}(\pi_0).\]
Since there is a square integrable representation $T_2$ such that  $T_1\hookrightarrow \chi_W\nu^{\frac{l-1}{2}}\rtimes T_2,$ where $T_1$ is the same square-integrable representation as in the previous case, we get that $T_2=\theta_{-(l-2)}^{down}(\pi_0).$ Now, if $l>3,$ $\theta_{-(l-2)}^{down}(\pi_0)$ is non-tempered, and we got a contradiction. If $l=3,$ we have to argue a bit more. Recall that in $l=3$--case, we have the following $\chi_V$-- or $\chi_W$--parts of the L-parameters of the corresponding tempered representations
\begin{align*}
&\pi \ldots \overset{-\epsilon}{\chi_V\otimes S_1}\oplus \overset{\epsilon}{\chi_V\otimes S_3}\oplus \overset{\epsilon}{\chi_V\otimes S_5}\oplus (*),\\ 
&\pi'\ldots \overset{-\epsilon}{\chi_V\otimes S_1}\oplus 2 \overset{\epsilon}{\chi_V\otimes S_3}\oplus\qquad \oplus  (*),\\ 
&\xi_0 \ldots \overset{-\epsilon_1}{\chi_W\otimes S_1}\oplus 2\overset{\epsilon_1}{\chi_W\otimes S_3}\oplus \overset{-\epsilon_1}{\chi_W\otimes S_5}\oplus (**),\\
&\pi_0\ldots \overset{-\epsilon}{\chi_V\otimes S_1}\oplus\qquad\oplus \qquad \oplus (*).
\end{align*}
Here,  if it is non-zero, $(*)$ denotes the same summands (with the the same $\epsilon_{\pi}$) with $\epsilon_{\pi}$ alternating on them, without gaps.  Recall that
$\xi_0\hookrightarrow \chi_W\delta[\nu^{-1},\nu^{1}]\rtimes T_0,$ and by \eqref{eq:0},
we have
\[\chi_W\delta[\nu^{-1},\nu^{1}]\otimes T_0\le \mu^*(\chi_W\nu^2\rtimes \theta_{-3}(\pi')).\]
This means that for some irreducible $\xi_2',$ we have $\chi_W\delta[\nu^{-1},\nu^{1}]\otimes \xi_2'\le \mu^*(\theta_{-3}(\pi'))$ and, then,
\[T_0\le \chi_W\nu^2\rtimes \xi_2'.\]
As before, since $T_0\hookrightarrow \chi_W\nu^2\rtimes T_1,$ we get that $T_1=\xi_2'$
and we want to derive a contradiction here.
Now, since $\theta_{-3}(\pi')\hookrightarrow \chi_W\delta[\nu^{-1},\nu^{1}]\rtimes \theta_{-3}(\pi_0),$ using \eqref{eq_tadic_classical}, we easily get
\[\xi_2'\le \chi_W\nu^1\rtimes \theta_{-1}^{down}(\pi_0).\]
Now we assume the contrary, i.e.~$\xi_2'=T_1$ is  square-integrable, so that 
\[\chi_W\delta[\nu^{-1},\nu^{1}]\otimes T_1 \le \mu^*(\theta_{-3}(\pi')).\]
Note that, then, $\chi_W\nu^1\otimes  \theta_{-1}^{down}(\pi_0)\le \mu^*(\xi_2').$
By using Frobenius, we get
\[\mu^*(\theta_{-3}(\pi'))\hookrightarrow \chi_W\nu^1\times\chi_W\nu^0\times\chi_W\nu^{-1}\times\chi_W\nu^1\rtimes \lambda\cong \chi_W\nu^1\times\chi_W\nu^0\times\chi_W\nu^{1}\times\chi_W\nu^{-1}\rtimes \lambda,\]
for some irreducible representation $\lambda.$ We have already concluded that $Jac_{\chi_W\nu^{\frac{l-1}{2}}}(\theta_{-l}(\pi'))$ is irreducible and equal to $\theta_{-l}(\pi''),$ so, again using Frobenius together with Lemme III.3 from the second chapter of \cite{MVW_theta}, we get
\[\theta_{-3}(\pi'')\hookrightarrow \chi_W\nu^0\times\chi_W\nu^{1}\rtimes \lambda',\]
for some irreducible $\lambda'.$ When we apply Lemma \ref{AG-epimorphism}, we get
\[\pi''\hookrightarrow \chi_V\nu^0\times\chi_V\nu^{1}\rtimes \lambda''\]
for some irreducible $\lambda''.$

\smallskip

 We now prove that this is impossible; namely we prove that $Jac_{\chi_V\nu^0}(\pi'')$ is irreducible and equal to $L(\chi_V\nu^1;\pi_0);$ moreover, $\chi_V\nu^1\rtimes \pi_0$ is reducible, and $L(\chi_V\nu^1;\pi_0)$ does not embed in $\chi_V\nu^{1}\rtimes \tau$ for any $ \tau.$ Indeed, note that in the appropriate Grothendieck group we have
\[\chi_V\delta[\nu^0,\nu^1]\rtimes \pi_0+\chi_V\zeta(\nu^0,\nu^1)\rtimes \pi_0=\chi_V\nu^0\rtimes L(\chi_V\nu^1;\pi_0)+\chi_V\nu^0\rtimes St_{\chi_V,\pi_0},\]
where by $ St_{\chi_V,\pi_0}$ we denote the unique irreducible subrepresentation of 
$\chi_V\nu^1\rtimes \pi_0$ (which is square-integrable). Indeed, using reasoning similar to already used above (as in \cite{Muic_positive}), we can prove that the representation $\chi_V\nu^1\rtimes \pi_0$ is of length 2 so that the above equality (in the Grothendieck group) becomes obvious. We know that the representation $\chi_V\nu^0\rtimes St_{\chi_V,\pi_0}$ is of length two (e.g.~\cite{MT})-the sum of the two tempered representations. The length of  $\chi_V\nu^0\rtimes L(\chi_V\nu^1;\pi_0)$ is at most two (we calculate $Jac_{\chi_V\nu^0}(\chi_V\nu^0\rtimes L(\chi_V\nu^1;\pi_0))$ and note that $\chi_V\nu^0\rtimes L(\chi_V\nu^1;\pi_0)$ is semisimple). It is of length two, since $L(\chi_V\delta[\nu^0,\nu^1]; \pi_0)+L(\chi_V\nu^1;\chi_V\nu^0\rtimes\pi_0)\le \chi_V\nu^0\rtimes L(\chi_V\nu^1;\pi_0).$ We get that $Jac_{\chi_V\nu^0}(\chi_V\nu^0\rtimes L(\chi_V\nu^1;\pi_0))=2L(\chi_V\nu^1;\pi_0)$ so that $Jac_{\chi_V\nu^0}(\pi'')=Jac_{\chi_V\nu^0}(L(\chi_V\delta[\nu^0,\nu^1]; \pi_0))=L(\chi_V\nu^1;\pi_0),$ as needed.

 \end{proof}

\begin{prop}
Under the induction assumption,  $\Theta_{-l}(\pi)$ does not have a non-tempered  subquotient different from the small theta lift $\theta_{-l}(\pi)$ .
\end{prop}

\begin{proof} By Theorem 4.1 of \cite{Muic_Israel}, if $\xi_0$ is a  non-tempered  subquotient of $\Theta_{-l}(\pi),$ then
\[\xi_0\hookrightarrow \chi_W\nu^{-\frac{l-1}{2}}\rtimes \xi_1,\]
where $\xi_1\le \Theta_{-(l-2)}(\pi).$ We have proved that $\Theta_{-(l-2)}(\pi)$ does not have tempered subquotients, thus
\[\xi_1\hookrightarrow \chi_W\nu^{-\frac{l-3}{2}}\rtimes \xi_2,\]
where $\xi_2\le \Theta_{-(l-4)}(\pi),$ etc. To conclude, we have
\[\xi_0\hookrightarrow \chi_W\nu^{-\frac{l-1}{2}}\times \chi_W\nu^{-\frac{l-3}{2}}\times \cdots \times \chi_W\nu^{-1}\rtimes \theta_{-1}^{down}(\pi)\]
or
\[\xi_0\hookrightarrow \chi_W\nu^{-\frac{l-1}{2}}\times \chi_W\nu^{-\frac{l-3}{2}}\times \cdots \times \chi_W\nu^{-\frac{l(\pi)+3}{2}}\rtimes \theta_{-(l(\pi)+2)}^{up}(\pi).\]
The representation on the right-hand side of the inclusions are socle--irreducible so that
\[\xi_0=\theta_{-l}(\pi),\]
cf.~\cite{AG_tempered}.
\end{proof}

The only thing which remained to be proved is that $\theta_{-l}(\pi)$ (which is, as we concluded above, the only irreducible subquotient of  $\Theta_{-l}(\pi)$)
appears with the multiplicity one in  $\Theta_{-l}(\pi).$ We were unable to find the exact reference for this claim, so we give (an easy) proof.

\begin{lem} We retain the notation of the previous propositions i.e. $\pi$ is a square-integrable representation. Then,  $\theta_{-l}(\pi)$ appears with the multiplicity one in $\Theta_{-l}(\pi).$ 
\end{lem}

\begin{proof} Assume that there exist two consecutive elements in $\Jord_{\rho}(\pi),$ say $2l_1+1$ and $2l_2+1$  such that $\epsilon_{\pi}(\rho\otimes S_{2l_1+1})=\epsilon_{\pi}(\rho\otimes S_{2l_2+1})$ and let $\pi_0$ be a square integrable representation such that
\[\pi\hookrightarrow \chi_V\delta[\rho\nu^{-l_1},\rho\nu^{l_2}]\rtimes \pi_0.\]
By Lemma \ref{AG-epimorphism} and the induction hypothesis we have 
\[\widetilde{\Theta_{-l}(\pi)}^{\eta}\hookrightarrow \chi_W\delta[\rho\nu^{-l_1},\rho\nu^{l_2}]\rtimes \theta_{-l}(\pi_0).\]
We calculate the multiplicity of $\chi_W\nu^{-\frac{l-1}{2}}\otimes \theta_{-(l-2)}(\pi) \in \mu^*(\chi_W\delta[\rho\nu^{-l_1},\rho\nu^{l_2}]\rtimes \theta_{-l}(\pi_0)).$ Using \eqref{eq_tadic_classical},\eqref{eq_tadic_classical2}, it turns out that the multiplicity is one; this, in turn, means that the multiplicity of $\theta_{-l}(\pi)$ in $\Theta_{-l}(\pi)$ is one.

Assume now that, for any $\rho,$ the function $\epsilon_{\pi}$ alternates on $\Jord_{\chi_V\rho}(\pi)$ ;
in the parlance of \cite{MT} and \cite{Muic_positive}, this means that $\pi$ is a strongly positive discrete series. Assume that for some  positive integer $\beta$ (if $\rho\neq 1$)  or $\beta\neq l,l+2$ if $\rho=1$ we have $\pi':=Jac_{\chi_V\rho\nu^{\frac{\beta-1}{2}}}(\pi)\neq 0.$ Note that $\pi'$ is necessarily square integrable.
Then we have
\[\widetilde{\Theta_{-l}(\pi)}^{\eta}\hookrightarrow \chi_W\rho\nu^{\frac{\beta-1}{2}}\rtimes \theta_{-l}(\pi'),\]
and we easily get that the multiplicity  of $\chi_W\nu^{-\frac{l-1}{2}}\otimes \theta_{-(l-2)}(\pi)$ in $\mu^*(\chi_W\rho \nu^{\frac{\beta-1}{2}}\rtimes \theta_{-l}(\pi'))$ is one, since the multiplicity of $\chi_W\nu^{-\frac{l-1}{2}}\otimes \theta_{-(l-2)}(\pi')$ in $\mu^*(\theta_{-l}(\pi'))$ equals one.

\smallskip
Assume now that $\beta=l+2$ and $\rho=1.$ Let $\omega$ be the relevant Weil representation (restriction of which gives us $\pi\otimes \Theta_{-l}(\pi)$). 
We have
\[\omega\twoheadrightarrow \pi\otimes  \Theta_{-l}(\pi),\]
so that we have
\[R_{P_1}(\omega)\twoheadrightarrow \chi_V\nu^{\frac{l+1}{2}}\otimes \pi'\otimes  \Theta_{-l}(\pi),\]
thus
\[\Theta(\chi_V\nu^{\frac{l+1}{2}}\otimes \pi',R_{P_1}(\omega))\twoheadrightarrow\Theta_{-l}(\pi).\]
Now we carefully examine the filtration of $\Theta_0:=\Theta(\chi_V\nu^{\frac{l+1}{2}}\otimes \pi',R_{P_1}(\omega))$ given in Theorem 3.1 of \cite{Muic_Israel} and incorporate the induction assumption (for $\pi'$).
We have 
\[\{0\}\subset \Theta_1\subset \Theta_0,\]
where $\chi_W\nu^{-\frac{l+1}{2}}\rtimes \theta_{-l}(\pi')\twoheadrightarrow \Theta_1$
and $\theta_{-l-2}(\pi')\twoheadrightarrow \Theta_0/\Theta_1.$ Since $\theta_{-l-2}(\pi')\neq \theta_{-l}(\pi)$ (our assumption is, from the beginning, that $\theta_{-l}(\pi)$ is non-tempered), we get
\[\Theta_{-l}(\pi)\le \chi_W\nu^{-\frac{l+1}{2}}\rtimes \theta_{-l}(\pi').\]
We easily get that the multiplicity of $\chi_W\nu^{-\frac{l-1}{2}}\otimes \theta_{-(l-2)}(\pi)$ in $\mu^*(\chi_W\nu^{-\frac{l+1}{2}}\rtimes \theta_{-l}(\pi'))$ is equal to one, since the multiplicity of $\chi_W\nu^{-\frac{l-1}{2}}\otimes \theta_{-(l-2)}(\pi')$ in $\mu^*(\theta_{-l}(\pi'))$ equals  one.

\smallskip
Now we just have to check the situation $\beta=l$ and $\rho=1.$ Lemma \ref{AG-epimorphism} gives
\begin{equation}
\label{last_epi}
\chi_W\nu^{-\frac{l-1}{2}}\rtimes \theta_{-l}(\pi')\twoheadrightarrow \Theta_{-l}(\pi).
\end{equation}
On the other hand, we have 
\[\theta_{-l}(\pi)\hookrightarrow \chi_W\nu^{-\frac{l-1}{2}}\rtimes \theta_{-(l-2)}(\pi)\hookrightarrow \chi_W\nu^{-\frac{l-1}{2}}\times \chi_W\nu^{\frac{l-1}{2}}\rtimes \theta_{-(l-2)}(\pi'). \]
Here the second embedding $ \theta_{-(l-2)}(\pi)\hookrightarrow \chi_W\nu^{\frac{l-1}{2}}\rtimes \theta_{-(l-2)}(\pi')$ is  the previous case, just shifted-we have used  here that $\theta_{-(l-2)}(\pi)\neq \theta_{-l}(\pi')$ (we are going to comment on that later).
Now we just calculate the multiplicity of 
\[\chi_W\nu^{-\frac{l-1}{2}}\times \chi_W\nu^{\frac{l-1}{2}}\otimes  \theta_{-(l-2)}(\pi')\]
in $\mu^*(\chi_W\nu^{-\frac{l-1}{2}}\rtimes \theta_{-l}(\pi')).$ We get that it equals one, as needed.

Now we comment on the relation between the lifts $\theta_{-(l-2)}(\pi)$ and  $\theta_{-l}(\pi').$ If both of these lifts are non-tempered, they are obviously non-equal and each step of the above reasoning is  valid. So we analyze when both of them are tempered.
Thus, we have two cases
\begin{itemize}
\item[(i)] $l(\pi)=-1.$\\ 

In this case, in order for 	 $\theta_{-(l-2)}(\pi)$ to be tempered, we have $l-2=1,$ and $\beta=l=3.$ Then, $l(\pi')=1,$ and, on exactly one tower, we have $\theta_{-1}(\pi)=\theta_{-3}^{up}(\pi').$ Note that $\epsilon_{\theta_{-3}^{up}(\pi')}(\chi_W\otimes S_1)=-\epsilon_{\theta_{-3}^{up}(\pi')}(\chi_W\otimes S_3).$
Then, we can conclude that $\Theta_{-l}(\pi)$ is irreducible just from \eqref{last_epi}; namely, the representation
\[\chi_W\nu^{-\frac{l-1}{2}}\rtimes \theta_{-l}(\pi')\]
is irreducible by Section 5 of \cite{MT}.
\item[(ii)] $l(\pi)\ge 1.$\\ 
 In case that we examine $\theta_{-l}^{down}(\pi),$ then $\theta_{-(l-2)}^{down}(\pi)$ is tempered only if $l=3,$ but then (recall that $\beta=l=3$) then $Jac_{\chi_V\nu^1}(\pi)=0,$ a contradiction (recall that $\pi$ is strongly positive).
 Thus, we must have $\theta_{-(l-2)}^{up}(\pi)$ is tempered, so $l=l(\pi)+4.$
 Note that the fact that $\pi$ is strongly positive means that $l(\pi)+2\notin\Jord_{\chi_V}(\pi).$ We easily get that $\theta_{-(l(\pi)+2)}^{up}(\pi)=\theta_{-(l(\pi)+4)}^{up}(\pi').$ Again, then the representation $\chi_W\nu^{-\frac{l-1}{2}}\rtimes \theta_{-l}(\pi')$ is irreducible, so is $\Theta_{-l}^{up}(\pi),$ as claimed.
\end{itemize}
\end{proof}
We have proved Theorem \ref{going_down_tower}.

\bigskip

Now we can pick some low--hanging fruit; actually, we have used this fact several times in the proof above, but we note it here again.

\begin{cor}
Let $l\ge 1$ be an odd  positive integer and $\pi$ a tempered representation. Then, if $\Theta_{-l}(\pi)$ is non-zero, then it is an irreducible representation.
\end{cor}
 \begin{proof}
 Let $l_1,\ldots,l_k$ be positive integers, and $\rho_1,\ldots,\rho_k$ unitarizable supercuspidal representations of $\GL(n_1,F),\ldots, \GL(n_k,F)$ and an irreducible  square--integrable representation $\sigma$ such that
 \[\chi_V\delta[\rho_1\nu^{-\frac{l_1-1}{2}},\rho_1\nu^{\frac{l_1-1}{2}}]\times \cdots \times \chi_V\delta[\rho_k\nu^{-\frac{l_k-1}{2}},\rho_k\nu^{\frac{l_k-1}{2}}]\rtimes \sigma \twoheadrightarrow \pi.\]
 Now we can use Lemma \ref{AG-epimorphism} and Theorem \ref{going_down_tower} to obtain
 \[\chi_W\delta[\rho_1\nu^{-\frac{l_1-1}{2}},\rho_1\nu^{\frac{l_1-1}{2}}]\times \cdots \times \chi_W\delta[\rho_k\nu^{-\frac{l_k-1}{2}},\rho_k\nu^{\frac{l_k-1}{2}}]\rtimes \theta_{-l}(\sigma) \twoheadrightarrow \Theta_{-l}(\pi).\]
 Since the left-hand side is semisimple, so is $\Theta_{-l}(\pi),$ and by Howe duality conjecture, $\Theta_{-l}(\pi)=\theta_{-l}(\pi),$ as claimed. 
 \end{proof}

 We now address the question of reducibility of $\Theta_{l}(T),$ where $l>0$ and $T$ tempered.

 \begin{prop}
 \label{Big_theta_temp_low}
 Let $T$ be an irreducible tempered representation of $G_n$ and let $l>0$ be such that $\Theta_{l}(T)\neq 0.$ Then, all the irreducible subquotients of $\Theta_{l}(T)$ are tempered.
 \end{prop}

 \begin{proof}
 Assume the contrary, i.e.~there exists a non-tempered subquotient $\xi\le \Theta_{l}(T);$  then, there exists a supercuspidal representation $\rho,$  real numbers $\alpha, \beta$ and irreducible representation $\xi'$ such that
 \[\xi\hookrightarrow \chi_W\delta[\rho\nu^{-\alpha},\rho\nu^{\beta}]\rtimes \xi'\]
 with $\alpha +\beta\in\Z_{\ge 0}$ and $-\alpha+\beta <0.$
 Using the projectivity of cuspidal representations together with the transitivity of Jacquet modules, we get that there exists an irreducible representation $\tau$ such that
 \[\Hom_{H_m}(\Theta_l(T),\chi_W\rho\nu^{\beta}\times \chi_W\rho\nu^{\beta-1}\times \cdots \times \chi_W\rho\nu^{-\alpha}\rtimes \tau)\neq 0.\]
 From this, again by using the procedure of Muić (\cite{Muic_Israel}, Theorem 4.1), we get that there exists $\gamma$ such that $-\alpha\le \gamma\le \beta$ and an irreducible representation $\tau_2$ such that
 \[Hom_{H_m}(\Theta_l(T),\chi_W\delta[\rho\nu^{-\alpha},\rho\nu^{\gamma}]\rtimes \tau_2)\neq 0.\]
 We now analyze $\Theta(\chi_W\delta[\rho\nu^{-\alpha},\rho\nu^{\gamma}]\otimes \tau_2,R_{Q_k}(\omega_{m,n}));$ note that 
 \[\Theta(\chi_W\delta[\rho\nu^{-\alpha},\rho\nu^{\gamma}]\otimes \tau_2,R_{Q_k}(\omega_{m,n}))\twoheadrightarrow T.\]
  We use Corollary 3.2 of \cite{Muic_Israel}; note that the roles of $m$ and $n$ here are inverted from the situation  in that Corollary (besides being the dimensions here and split ranks there). Since $\alpha>0$ and $l>0$ we get that the exceptional situation of that Corollary cannot occur, so that we get
 \[T\hookrightarrow\chi_V\delta[\rho\nu^{-\alpha},\rho\nu^{\gamma}]\rtimes \widetilde{\Theta_{-l}(\tau_2)}^{\eta},\]
 a contradiction with the temperedness of $T.$
 \end{proof}

\begin{prop}
\label{temp_odd}
Assume that  $T$ is an irreducible tempered representation of $G_n$ and $l>0$ such that $\Theta_{l}(T)\neq 0.$ Note that the multiplicity of $\chi_V\otimes S_{l}$ in $\phi_T$ is then always positive. If the multiplicity of $\chi_V\otimes S_{l}$ in $\phi_T$ equals one, $\Theta_{l}(T)=\theta_l(T).$ On the other hand, if this multiplicity  is odd, but greater than one, $\Theta_{l}(T)$ is reducible. 
\end{prop}

\begin{proof}
First assume that the multiplicity of $\chi_V\otimes S_{l}$ in $\phi_T$ equals one.
Let $l_1,\ldots,l_k$ be positive integers, and $\rho_1,\ldots,\rho_k$ unitarizable supercuspidal representations of $\GL(n_1,F),\ldots, \GL(n_k,F)$ and an irreducible  square--integrable representation $\sigma$ such that
 \[\chi_V\delta[\rho_1\nu^{-\frac{l_1-1}{2}},\rho_1\nu^{\frac{l_1-1}{2}}]\times \cdots \times \chi_V\delta[\rho_k\nu^{-\frac{l_k-1}{2}},\rho_k\nu^{\frac{l_k-1}{2}}]\rtimes \sigma \twoheadrightarrow T.\]
 Since the multiplicity is one, we have $(l,\chi_V)\notin \{(l_1,\chi_V\rho_1),\ldots,(l_k,\chi_V\rho_k)\}$ and we can apply Lemma \ref{AG-epimorphism} to get 
 \[\chi_W\delta[\rho_1\nu^{-\frac{l_1-1}{2}},\rho_1\nu^{\frac{l_1-1}{2}}]\times \cdots \times \chi_W\delta[\rho_k\nu^{-\frac{l_k-1}{2}},\rho_k\nu^{\frac{l_k-1}{2}}]\rtimes \Theta_{l}(\sigma) \twoheadrightarrow \Theta_l(T).\]
 Since $\Theta_{l}(\sigma)=\theta_{l}(\sigma)$ by Theorem \ref{going_down_tower}, the right-hand side is semisimple and the claim follows.

 \smallskip
Assume now that the multiplicity in question equals $2h+1,\;h\ge 1.$
Define a representation $\pi=L(\chi_V\delta[\nu^{-\frac{l-1}{2}},\nu^{\frac{l+1}{2}}];T).$ We know by \cite{BH_theta} that $l(\pi)=l(T)$ if $l<l(T)$ or $l(\pi)=l(T)+2$ if $l=l(T).$ Since we have $\chi_V\delta[\nu^{-\frac{l-1}{2}},\nu^{\frac{l+1}{2}}]\rtimes T\twoheadrightarrow\pi,$ by the second part of \ref{AG-epimorphism} with respect  to $\Theta_{l+2}$ we get that either
\[\chi_W\delta[\nu^{-\frac{l-1}{2}},\nu^{\frac{l-1}{2}}]\rtimes \Theta_{l}(T)\twoheadrightarrow \theta_{l+2}(\pi)\]
or 
\[\chi_W\delta[\nu^{-\frac{l-1}{2}},\nu^{\frac{l+1}{2}}]\rtimes \Theta_{l+2}(T)\twoheadrightarrow \theta_{l+2}(\pi).\]
We know that $ \theta_{l+2}(\pi)$ is tempered (cf.~Lemma 5.5 and Theorem 6.1 of \cite{BH_theta}).
The second possibility cannot occur-indeed, if $l=l(T)$, then $\Theta_{l+2}(T)=0,$ and if $l<l(T)$ let $\xi\le \Theta_{l+2}(T)$ be an irreducible subquotient which contributes in the epimorphism just above. By Proposition \ref{Big_theta_temp_low}, $\xi$ is tempered. Then, $\chi_W\delta[\nu^{-\frac{l-1}{2}},\nu^{\frac{l+1}{2}}]\rtimes \xi$ is a standard representation which cannot have a tempered quotient $ \theta_{l+2}(\pi)$.

Thus, we necessarily have 
\begin{equation}
\label{embedding_wrong}
\theta_{l+2}(\pi)\hookrightarrow \chi_W\delta[\nu^{-\frac{l-1}{2}},\nu^{\frac{l-1}{2}}]\rtimes \widetilde{\Theta_{l}(T)}^{\eta}.
\end{equation}
 Assume that, in this embedding, $\theta_{l}(T)$ contributes. Note that the representation 
 \[\chi_W\delta[\nu^{-\frac{l-1}{2}},\nu^{\frac{l-1}{2}}]\rtimes \theta_{l}(T)\] is irreducible, the multiplicity of $\chi_W\otimes S_l$ here equals $2h+2$ and the character  satisfies $\epsilon_{\theta_{l}(T)}(\chi_W\otimes S_l)=\epsilon_{T}(\chi_V\otimes S_l).$
 On the other hand, by the arguments in Lemma 5.5 of \cite{BH_theta}, also cf.~Theorem 4.5. (2) of \cite{AG_tempered}, we know that $\theta_{l+2}(\pi)$ is tempered representation with the character on $\chi_W\otimes S_l$  opposite to $\epsilon_{T}(\chi_V\otimes S_l).$ This means that in the \eqref{embedding_wrong},  $\theta_{l}(T)$ does not contribute, i.e. $\Theta_{l}(T)$ is not irreducible.
\end{proof}

Now we discuss the last possibility for the multiplicities of $\chi_V\otimes S_l$ in $\phi_T.$
\begin{prop}
\label{temp_odd}
Assume that  $T$ is an irreducible tempered representation of $G_n$ and $l>0$ such that $\Theta_{l}(T)\neq 0.$  If the multiplicity of $\chi_V\otimes S_{l}$ in $\phi_T$ is even, $\Theta_{l}(T)=\theta_l(T);$ i.e.~$\Theta_{l}(T)$ is irreducible. 
\end{prop}

\begin{proof}Note that, in the situation from the statement of this Proposition, we necessarily have $l=l(T)$ (cf.~\cite{AG_tempered}, Theorem 4.1.(1)).
Thus, we have
\[\phi_T\ldots m_1\chi_V\otimes S_1\oplus m_3\chi_V\otimes S_3\oplus \cdots\oplus m_{l(T)-2}\chi_V\otimes S_{l(T)-2}\oplus 2h\chi_V\otimes S_{l(T)}\oplus\cdots;\]
here the multiplicities $m_1,m_3,\ldots,m_{l(T)-2}$ are odd.

Recall that $T$ can be embedded in the unitarizable induced representation
\[T\hookrightarrow \chi_V\delta_1^{n_1}\times \cdots \times \chi_V\delta_k^{n_k}\rtimes \sigma,\]
where $\delta_1,\ldots,\delta_k$ are  mutually non-isomorphic irreducible discrete series representations, $\delta_i^{n_i}$ denotes the induced representation $\delta_i\times \delta_i\times \cdots\times \delta_i$ ($n_i$ times). Here $\sigma$ is ``almost'' discrete series, i.e.~it is a tempered representation with the multiplicities of summands in $\phi_{\sigma}$  always equal to one, except we have $2h \chi_V\otimes S_{l(T)}$ appearing. Note that $l(\sigma)=l(T).$

First we prove the statement of the Proposition for the representation $\sigma.$
We introduce the auxiliary representations $\sigma_i,\;i=1,2,\ldots h$ which have the same L-parameter and $\epsilon$ function except $\sigma_i$ has $2i\chi_V\otimes S_{l(T)}$ appearing in the L-parameter with $\epsilon_{\sigma_i}(2i\chi_V\otimes S_{l(T)})=\epsilon_{T}(2h\chi_V\otimes S_{l(T)})$ so that our  $\sigma$ is $\sigma_{h}.$
We denote by $\sigma_0$ a discrete series representation such that
\[\sigma_i\hookrightarrow \chi_V\delta[\nu^{-\frac{l(T)-1}{2}},\nu^{\frac{l(T)-1}{2}}]^i\rtimes \sigma_0.\]
To prove the statement for $\sigma=\sigma_h,$ we proceed inductively and   start with the basic case i.e., prove that $\Theta_{l(T)}(\sigma_1)=\theta_{l(T)}(\sigma_1).$
Using filtration of the Corollary 3.2 of \cite{Muic_Israel} with the fact that $l(\sigma_0)=l(T)-2,$
 we get 
 \[\chi_W\delta[\nu^{-\frac{l(T)-1}{2}}, \nu^{\frac{l(T)-3}{2}}]\rtimes \Theta_{l(T)-2}(\sigma_0)\twoheadrightarrow \Theta_{l(T)}(\sigma_1).\]
 Since $\sigma_0$ is square-integrable, $\Theta_{l(T)-2}(\sigma_0)=\theta_{l(T)-2}(\sigma_0).$ Note that the representation on the left hand side is of length three; the unique irreducible subrepresentation is the corresponding Langlands quotient. Note that it has to be in the kernel of the above epimorphism, since $\Theta_{l(T)}(\sigma_1)$ has only tempered subquotients. Thus, we have an epimorphism from the cosocle of
 $\chi_W\delta[\nu^{-\frac{l(T)-1}{2}}, \nu^{\frac{l(T)-3}{2}}]\rtimes \Theta_{l(T)-2}(\sigma_0)$, which is a direct sum of two square-integrable representations, onto 
$\Theta_{l(T)}(\sigma_1).$ Thus, $\Theta_{l(T)}(\sigma_1)=\theta_{l(T)}(\sigma_1)$
(actually, in this case we could  also argue that $\theta_{l(T)}(\sigma_1)$ is square-integrable, so all the subquotients of $\Theta_{l(T)}(\sigma_1)$  are also square-integrable, and the projectivity of square-integrable representations in the appropriate category gives the same result, but we use the previous argument also later).

Now we proceed inductively; the case of $\sigma_2$ is  still little bit less technical but has all the steps as the general induction step, so to explain the procedure we  go through the case of $\sigma_2$ in detail.\\ 

Steps:\\
\noindent 1. The representation 
\[\chi_V\delta[\nu^{-\frac{l(T)-3}{2}},\nu^{\frac{l(T)-1}{2}}]\rtimes \sigma_1\]
is irreducible. This follows from \cite{Muic_standard}, Introduction-the third elementary case, also cf.~Lemma 6.4 there. We have
\begin{align*}
&\sigma_2=\chi_V\delta[\nu^{-\frac{l(T)-1}{2}},\nu^{\frac{l(T)-1}{2}}]\rtimes \sigma_1\hookrightarrow \chi_V\nu^{\frac{l(T)-1}{2}}\times \chi_V\delta[\nu^{-\frac{l(T)-1}{2}},\nu^{\frac{l(T)-3}{2}}]\rtimes \sigma_1\cong\\
&\chi_V\nu^{\frac{l(T)-1}{2}}\times \chi_V\delta[\nu^{-\frac{l(T)-3}{2}},\nu^{\frac{l(T)-1}{2}}]\rtimes \sigma_1.
\end{align*}
Thus, using Lemma \ref{AG-epimorphism} and the fact that $\Theta_{l(T)}(\sigma_1)=\theta_{l(T)}(\sigma_1),$ we get
\[\chi_W\nu^{-\frac{l(T)-1}{2}}\times \chi_W\delta[\nu^{-\frac{l(T)-1}{2}},\nu^{\frac{l(T)-3}{2}}]\rtimes \theta_{l(T)}(\sigma_1)\twoheadrightarrow \Theta_{l(T)}(\sigma_2)\twoheadrightarrow \theta_{l(T)}(\sigma_2).\]
\noindent 2. We want to prove, using the epimorphism just above, that
 $\theta_{l(T)}(\sigma_2)$ appears with the multiplicity one in the composition series of $\Theta_{l(T)}(\sigma_2).$ We shall do that by examining the multiplicity with which a certain irreducible representation appears in the appropriate Jacquet modules. Namely, we prove that 
 \[\chi_W\nu^{\frac{l(T)-1}{2}}\times \chi_W\delta[\nu^{-\frac{l(T)-3}{2}},\nu^{\frac{l(T)-1}{2}}]\otimes \theta_{l(T)}(\sigma_1)\]
 appears $x$ times in the appropriate Jacquet module of $\theta_{l(T)}(\sigma_2)$
 and less then $2x$ times in the appropriate Jacquet module of $\Theta_{l(T)}(\sigma_2).$ In this case, we prove that $\chi_W\nu^{\frac{l(T)-1}{2}}\times \chi_W\delta[\nu^{-\frac{l(T)-3}{2}},\nu^{\frac{l(T)-1}{2}}]\otimes \theta_{l(T)}(\sigma_1)$ appears with the multiplicity $6$ in the Jacquet module of $\theta_{l(T)}(\sigma_2)$ and with multiplicity $7$ in the Jacquet module of  $\chi_W\nu^{\frac{l(T)-1}{2}}\times \chi_W\delta[\nu^{-\frac{l(T)-3}{2}},\nu^{\frac{l(T)-1}{2}}]\rtimes \theta_{l(T)}(\sigma_1)$, so the multiplicity with which it appears in the appropriate Jacquet
 module of $\Theta_{l(T)}(\sigma_2)$, by the above epimorphism,  is less or equal to $7.$ The calculation uses \eqref{eq_tadic_classical},\eqref{eq_tadic_classical2}. To conclude, since $\theta_{l(T)}(\sigma_2)=\chi_W\delta[\nu^{-\frac{l(T)-1}{2}},\nu^{\frac{l(T)-1}{2}}]\rtimes \theta_{l(T)}(\sigma_1),$ we have to enumerate all the possibilities for the following occurrences 
 \[\chi_W\nu^{\frac{l(T)-1}{2}}\times \chi_W\delta[\nu^{-\frac{l(T)-3}{2}},\nu^{\frac{l(T)-1}{2}}]\le \chi_W\delta[\nu^{-i},\nu^{\frac{l(T)-1}{2}}]\times \chi_W\delta[\nu^{j+1},\nu^{\frac{l(T)-1}{2}}]\times \lambda_1\]
 and
 \[\theta_{l(T)}(\sigma_1)\le\chi_W\delta[\nu^{i+1},\nu^{j}]\rtimes \lambda_2, \]
 for some $-\frac{l(T)+1}{2}\le i\le j\le \frac{l(T)-1}{2},$ with  irreducible subquotient $\lambda_1\otimes \lambda_2\le \mu^*(\theta_{l(T)}(\sigma_1)).$
 Using properties of the generic representations of general linear groups, together with the properties of the L-parameter of $\theta_{l(T)}(\sigma_1),$ we conclude that there are $6$ such possibilities. 

 In this calculation, it is important to calculate the first derivative with respect of $\chi_W\nu^{\frac{l(T)-1}{2}}$ of 
 $\theta_{l(T)}(\sigma_1)$-since $\theta_{l(T)}(\sigma_1)$ is square-integrable, this derivative is an irreducible tempered representation, say $\lambda_0$, for which we indeed have 
 \[\theta_{l(T)}(\sigma_1)\le \chi_W\nu^{\frac{l(T)-1}{2}}\rtimes \lambda_0.\] Also, we want to find all $\xi$  such that 
 \[\chi_W\delta[\nu^{-\frac{l(T)-3}{2}},\nu^{\frac{l(T)-1}{2}}]\otimes \xi\le \mu^*(\theta_{l(T)}(\sigma_1)).\]
 It turns out that there is only one such $\xi,$ it is a discrete series representation and we know (\cite{MT}) that
 \[\theta_{l(T)}(\sigma_1))\le \chi_W\delta[\nu^{-\frac{l(T)-3}{2}},\nu^{\frac{l(T)-1}{2}}]\rtimes \xi.\]
 We emphasize these calculations of the Jacquet modules of $\theta_{l(T)}(\sigma_1),$
 because these calculations become more difficult when we calculate them for $\theta_{l(T)}(\sigma_{i-1}),\;i\ge 3.$

 \noindent 3. Now we apply the filtration of $\Theta:=\Theta(\chi_V\delta[\nu^{-\frac{l(T)-1}{2}},\nu^{\frac{l(T)-1}{2}}]\otimes \sigma_1, R_P(\omega_{m,n}))$ calculated in \cite{Muic_Israel}, which we have used before, so that
 \[\{0\}\subset \Theta_0\subset \Theta.\] 
 Here $\omega_{m,n}$ in the appropriate Weil representation which has the relevant quotient $\sigma_2\otimes \Theta_{l(T)}(\sigma_2).$ Recall that $\Theta\twoheadrightarrow \Theta_{l(T)}(\sigma_2);$ let $f$ denote this epimorphism.

 If $f_{|\Theta_0}\neq 0$  we have a non-zero mapping  
 \[\chi_W\delta[\nu^{-\frac{l(T)-1}{2}},\nu^{\frac{l(T)-1}{2}}]\rtimes \theta_{l(T)}(\sigma_1)\to \Theta_{l(T)}(\sigma_2).\]
 Recall that $\chi_W\delta[\nu^{-\frac{l(T)-1}{2}},\nu^{\frac{l(T)-1}{2}}]\rtimes \theta_{l(T)}(\sigma_1)=\theta_{l(T)}(\sigma_2),$ thus we have an embedding of $\theta_{l(T)}(\sigma_2)$ in $\Theta_{l(T)}(\sigma_2).$ In the previous step, we proved that the multiplicity of $\theta_{l(T)}(\sigma_2)$ in $\Theta_{l(T)}(\sigma_2)$ equals one, thus $\theta_{l(T)}(\sigma_2)=\Theta_{l(T)}(\sigma_2).$

 If, on the other hand, $f_{|\Theta_0}= 0,$ we have an epimorphism
 \[\chi_W\delta[\nu^{-\frac{l(T)-1}{2}},\nu^{\frac{l(T)-3}{2}}]\rtimes \Theta_{l(T)-2}(\sigma_1)\twoheadrightarrow \Theta_{l(T)}(\sigma_2).\] 
 By Proposition \ref{temp_odd}, $\Theta_{l(T)-2}(\sigma_1)=\theta_{l(T)-2}(\sigma_1).$ All in all, since $\chi_W\delta[\nu^{-\frac{l(T)-1}{2}},\nu^{\frac{l(T)-1}{2}}]\rtimes \theta_{l(T)-2}(\sigma_0)\twoheadrightarrow \theta_{l(T)-2}(\sigma_1)$ we have an epimorphism
 \[\chi_W\delta[\nu^{-\frac{l(T)-1}{2}},\nu^{\frac{l(T)-3}{2}}]\times \chi_W\delta[\nu^{-\frac{l(T)-1}{2}},\nu^{\frac{l(T)-1}{2}}]\rtimes \theta_{l(T)-2}(\sigma_0)\twoheadrightarrow \Theta_{l(T)}(\sigma_2).\]
 Since 
 \[\chi_W\delta[\nu^{-\frac{l(T)-1}{2}},\nu^{\frac{l(T)-3}{2}}]\times \chi_W\delta[\nu^{-\frac{l(T)-1}{2}},\nu^{\frac{l(T)-1}{2}}]\cong \chi_W\delta[\nu^{-\frac{l(T)-1}{2}},\nu^{\frac{l(T)-1}{2}}]\times \chi_W\delta[\nu^{-\frac{l(T)-1}{2}},\nu^{\frac{l(T)-3}{2}}],\]
 the following holds
 \[\chi_W\delta[\nu^{-\frac{l(T)-1}{2}},\nu^{\frac{l(T)-1}{2}}]\times \chi_W\delta[\nu^{-\frac{l(T)-1}{2}},\nu^{\frac{l(T)-3}{2}}]\rtimes  \theta_{l(T)-2}(\sigma_0)\twoheadrightarrow \Theta_{l(T)}(\sigma_2).\]
 The representation $\chi_W\delta[\nu^{-\frac{l(T)-1}{2}},\nu^{\frac{l(T)-1}{2}}]\times L(\chi_W\delta[\nu^{-\frac{l(T)-3}{2}},\nu^{\frac{l(T)-1}{2}}]; \theta_{l(T)-2}(\sigma_0))$ is a subrepresentation of the left-hand side and does not have a tempered subquotient. Indeed, this representation is unitarizable ($L(\chi_W\delta[\nu^{-\frac{l(T)-3}{2}},\nu^{\frac{l(T)-1}{2}}]; \theta_{l(T)-2}(\sigma_0))$ is unitarizable) and, consequently, semi-simple, so each subquotient, say $\lambda,$ is a subrepresentation. It follows
 \begin{align*}
 &\lambda\hookrightarrow \chi_W\delta[\nu^{-\frac{l(T)-1}{2}},\nu^{\frac{l(T)-1}{2}}]\times L(\chi_W\delta[\nu^{-\frac{l(T)-3}{2}},\nu^{\frac{l(T)-1}{2}}]; \theta_{l(T)-2}(\sigma_0))\\ 
 & \hookrightarrow \chi_W\delta[\nu^{-\frac{l(T)-1}{2}},\nu^{\frac{l(T)-1}{2}}]\times  \chi_W\delta[\nu^{-\frac{l(T)-1}{2}},\nu^{\frac{l(T)-3}{2}}]\rtimes \theta_{l(T)-2}(\sigma_0)\\ 
 &\cong \chi_W\delta[\nu^{-\frac{l(T)-1}{2}},\nu^{\frac{l(T)-3}{2}}]\times \chi_W\delta[\nu^{-\frac{l(T)-1}{2}},\nu^{\frac{l(T)-1}{2}}]\rtimes \theta_{l(T)-2}(\sigma_0).
 \end{align*}
 This means that $\chi_W\delta[\nu^{-\frac{l(T)-1}{2}},\nu^{\frac{l(T)-1}{2}}]\times L(\chi_W\delta[\nu^{-\frac{l(T)-3}{2}},\nu^{\frac{l(T)-1}{2}}]; \theta_{l(T)-2}(\sigma_0))$ is in the kernel of the above epimorphism since $\Theta_{l(T)}(\sigma_2)$ does not have non-tempered subquotients; it follows that the cosocle, which is a semisimple unitarizable representation
 \[\chi_W\delta[\nu^{-\frac{l(T)-1}{2}},\nu^{\frac{l(T)-1}{2}}]\rtimes(\lambda_1\oplus \lambda_2),\] maps onto $\Theta_{l(T)}(\sigma_2),$ guaranteeing that it is irreducible. Here $\lambda_i,\;i=1,2,$ are two discrete series subrepresentations of $\chi_W\delta[\nu^{-\frac{l(T)-3}{2}},\nu^{\frac{l(T)-1}{2}}]\rtimes \theta_{l(T)-2}(\sigma_0).$

 In these three steps we have proved that $\Theta_{l(T)}(\sigma_2)$ is irreducible.
 \smallskip

 Now assume that $\Theta_{l(T)}(\sigma_{i-1})$ is irreducible; we also assume that 
 the multiplicity of $\theta_{l(T)}(\sigma_{i-1})$ in the composition series of 
 \[\chi_W\nu^{-\frac{l(T)-1}{2}}\times \chi_W\delta[\nu^{-\frac{l(T)-1}{2}},\nu^{\frac{l(T)-3}{2}}]\rtimes \theta_{l(T)}(\sigma_{i-2})\]
 equals one.  We prove that $\Theta_{l(T)}(\sigma_{i})$ is irreducible. We go through the previous steps.

 \noindent 1. The representation 
\[\chi_V\delta[\nu^{-\frac{l(T)-3}{2}},\nu^{\frac{l(T)-1}{2}}]\rtimes \sigma_{i-1}\]
is irreducible.  Indeed, since $\sigma_{i-1}=\chi_V\delta[\nu^{-\frac{l(T)-1}{2}},\nu^{\frac{l(T)-1}{2}}]^{i-1}\rtimes \sigma_1$ and $\chi_V\delta[\nu^{-\frac{l(T)-3}{2}},\nu^{\frac{l(T)-1}{2}}]\times \chi_V\delta[\nu^{-\frac{l(T)-1}{2}},\nu^{\frac{l(T)-1}{2}}]$ is irreducible, by the reductions explained in the Introduction of \cite{Muic_standard}, this follows from the fact that 
\[\chi_V\delta[\nu^{-\frac{l(T)-3}{2}},\nu^{\frac{l(T)-1}{2}}]\rtimes \sigma_{1}\]
is irreducible. Now we have
\[\chi_W\nu^{-\frac{l(T)-1}{2}}\times \chi_W\delta[\nu^{-\frac{l(T)-1}{2}},\nu^{\frac{l(T)-3}{2}}]\rtimes \theta_{l(T)}(\sigma_{i-1})\twoheadrightarrow \Theta_{l(T)}(\sigma_i)\twoheadrightarrow \theta_{l(T)}(\sigma_i).\]

\noindent 2. We calculate the multiplicity of $\chi_W\nu^{\frac{l(T)-1}{2}}\times \chi_W\delta[\nu^{-\frac{l(T)-3}{2}},\nu^{\frac{l(T)-1}{2}}]\otimes \theta_{l(T)}(\sigma_{i-1})$ in the appropriate Jacquet module of $\theta_{l(T)}(\sigma_i)=\chi_W\delta[\nu^{-\frac{l(T)-1}{2}},\nu^{\frac{l(T)-1}{2}}]\rtimes \theta_{l(T)}(\sigma_{i-1}).$
This calculation is a bit more complicated than in the case $i=2.$ We give an example: in the process, by using \ref{eq_tadic_classical}, we have to find all the irreducible subquotients 
\[\chi_W\otimes \xi\le \mu^*(\theta_{l(T)}(\sigma_{i-1}))\]
such that
\[\theta_{l(T)}(\sigma_{i-1})\le \chi_W\rtimes \xi.\]
Let us denote $\delta_u=\chi_W\delta[\nu^{-\frac{l(T)-1}{2}},\nu^{\frac{l(T)-1}{2}}].$ Then, the possibilities are $\xi\in\{\delta_u^{i-3}\rtimes \xi_1, \delta_u^{i-3} \rtimes \xi_1', L(\chi_W\delta[\nu^{-\frac{l(T)-3}{2}},\nu^{\frac{l(T)-1}{2}}];\delta_u^{i-3}\rtimes \theta_{l(T)}(\sigma_1))\}.$ Here $\xi_1$ is a tempered representation with the the same L-parameter (and signs) as $\theta_{l(T)}(\sigma_1)$ except the summands $\chi_W\otimes S_{l(T)-2}$ and  $\chi_W\otimes S_{l(T)}$ come with the multiplicity 2 (they have the signs as in  $\theta_{l(T)}(\sigma_1)$) and $\xi_1'$ has the same L-parameter as $\xi_1,$ the same signs except it has the opposite sign on $\chi_W\otimes S_{l(T)}.$  We conclude that, out of all the options for $\xi,$ only $\delta_u^{i-3}\rtimes \xi_1$ satisfies the condition that
$\theta_{l(T)}(\sigma_{i-1})\le \chi_W\rtimes \xi.$ Indeed, $L(\chi_W\delta[\nu^{-\frac{l(T)-3}{2}},\nu^{\frac{l(T)-1}{2}}];\delta_u^{i-3}\rtimes \theta_{l(T)}(\sigma_1))$ does not satisfy it because of our induction assumption that $\theta_{l(T)}(\sigma_{i-1})$ appears with multiplicity one in the composition series of 
\[\chi_W\nu^{-\frac{l(T)-1}{2}}\times \chi_W\delta[\nu^{-\frac{l(T)-1}{2}},\nu^{\frac{l(T)-3}{2}}]\rtimes \theta_{l(T)}(\sigma_{i-2}).\]
On the other hand, $\delta_u^{i-3} \rtimes \xi_1'$ does not satisfy it because the highest derivative of $\chi_W\times \delta_u^{i-3} \rtimes \xi_1'$ with respect to $\chi_W\nu^{\frac{l(T)-1}{2}}$ is of order $2i-4,$ and of $\theta_{l(T)}(\sigma_{i-1})$ is of order $2i-3.$ Analogously, we have to find those irreducible subquotients
\[\chi_W\delta[\nu^{-\frac{l(T)-3}{2}},\nu^{\frac{l(T)-1}{2}}]\otimes \xi'\le \mu^*(\theta_{l(T)}(\sigma_{i-1}))\]
such that
\[\theta_{l(T)}(\sigma_{i-1})\le \chi_W\delta[\nu^{-\frac{l(T)-3}{2}},\nu^{\frac{l(T)-1}{2}}]\rtimes \xi'.\] We argue analogously.
All in all, we obtain that the multiplicity with which $\chi_W\nu^{\frac{l(T)-1}{2}}\times \chi_W\delta[\nu^{-\frac{l(T)-3}{2}},\nu^{\frac{l(T)-1}{2}}]\otimes \theta_{l(T)}(\sigma_{i-1})$ appears in the appropriate Jacquet module of $\theta_{l(T)}(\sigma_i)$ equals $8i-10.$ In the same vein, we get that the multiplicity with which $\chi_W\nu^{\frac{l(T)-1}{2}}\times \chi_W\delta[\nu^{-\frac{l(T)-3}{2}},\nu^{\frac{l(T)-1}{2}}]\otimes \theta_{l(T)}(\sigma_{i-1})$ appears in the appropriate Jacquet module of $\chi_W\nu^{\frac{l(T)-1}{2}}\times \chi_W\delta[\nu^{-\frac{l(T)-3}{2}},\nu^{\frac{l(T)-1}{2}}]\rtimes \theta_{l(T)}(\sigma_{i-1})$ equals $10i-13.$ Thus, $\theta_{l(T)}(\sigma_i)$ appears with multiplicity one in the composition series of $\chi_W\nu^{\frac{l(T)-1}{2}}\times \chi_W\delta[\nu^{-\frac{l(T)-3}{2}},\nu^{\frac{l(T)-1}{2}}]\rtimes \theta_{l(T)}(\sigma_{i-1})$ and so it appears with the multiplicity one in $\Theta_{l(T)}(\sigma_i).$ The third step is virtually the same as in the previous case of $i=2;$ thus, we have proved that $\Theta_{l(T)}(\sigma_i)=\theta_{l(T)}(\sigma_i)$ and, to return to the beginning of our proof, for  $\sigma=\sigma_h$ we have proved that  $\Theta_{l(T)}(\sigma)$ is irreducible. Now, since the representations $\chi_V\delta_1,\ldots,\chi_V\delta_k$ from the description of $T$ are all different from $\chi_V\delta[\nu^{-\frac{l(T)-1}{2}},\nu^{\frac{l(T)-1}{2}}],$ we can apply Lemma \ref{AG-epimorphism} to obtain that $\Theta_{l(T)}(T)$ is irreducible.

\end{proof}
\bibliographystyle{siam}
\bibliography{big_theta}
\end{document}